\documentclass[12pt]{amsart}
\usepackage{graphicx}
\usepackage{amssymb}
\usepackage{amsthm}
\usepackage{amsfonts}
\usepackage{amsmath}
\usepackage{multicol}
 \usepackage[curve,all]{xy}
\usepackage{color}

\DeclareGraphicsExtensions{.png,.jpg,.jpeg}

\newcommand{\bbR}{\mathbb{R}}
\newcommand{\bbC}{\mathbb{C}}
\newcommand{\bbZ}{\mathbb{Z}}
\newcommand{\bbP}{\mathbb{P}}

\newcommand{\bbQ}{\mathbb{Q}}
\newcommand{\bbH}{\mathbb{H}}

\newcommand{\rmO}{\textrm{O}}

\newcommand{\sfU}{\mathsf{U}}
\newcommand{\sfA}{\mathsf{A}}
\newcommand{\sfE}{\mathsf{E}}
\newcommand{\sfQ}{\mathsf{Q}}

\newcommand{\frake}{\mathfrak{e}}

\newcommand{\frakP}{\mathfrak{P}}
\newcommand{\frako}{\mathfrak{o}}

\newcommand{\bfk}{\mathbf{k}}

\newcommand{\calB}{\mathcal{B}}

\newcommand{\calH}{\mathcal{H}}
\newcommand{\calP}{\mathcal{P}}
\newcommand{\calS}{\mathcal{S}}

\newcommand{\calAp}{\mathcal{A}p}

\newcommand{\ev}{\textrm{ev}}
\newcommand{\Aut}{\textrm{Aut}}
\newcommand{\Num}{\textrm{Num}}
\newcommand{\Pic}{\textrm{Pic}}
\newcommand{\Or}{\textrm{O}}
\newcommand{\Cr}{\textrm{Cr}}
\newcommand{\la}{\langle}
\newcommand{\ra}{\rangle}
\newcommand{\Iso}{\textrm{Iso}}
\newcommand{\PO}{\textrm{PO}}
\newcommand{\SO}{\textrm{SO}}
\newcommand{\Mob}{\textrm{M\"ob}}
\newcommand{\GL}{\textrm{GL}}
\newcommand{\circulant}{\textrm{cir}}
\newcommand{\integ}{\textrm{int}}

\newcommand{\Ap}{\mathrm{Ap}}

\newcommand{\oh}{\overline{\bbH}}

\usepackage{amsthm}
\newtheorem{theorem}{Theorem}[section]
\newtheorem{lemma}[theorem]{Lemma}
\newtheorem{corollary}[theorem]{Corollary}

\theoremstyle{definition}

\newtheorem{example}[theorem]{Example}

\theoremstyle{remark}
\newtheorem{remark}[theorem]{Remark}

\numberwithin{equation}{section}

\newcommand{\bsm}{\left(\begin{smallmatrix}}
\newcommand{\esm}{\end{smallmatrix}\right)}

\newcommand{\beq}{\begin{equation}}
\newcommand{\eeq}{\end{equation}}

\title[Orbital counting]{Orbital counting of  curves on algebraic surfaces and  sphere packings}

\author{Igor Dolgachev}
\address{Department of Mathematics, University of Michigan, 525 E. University Av., Ann Arbor, MI, 49109}
\email{idolga@umich.edu}
\dedicatory{To the memory of Andrey Todorov}

\begin{document}

\begin{abstract} We realize the Apollonian group associated to an integral Apollonian circle packings, and some of its generalizations,  as a  group of automorphisms of an algebraic surface. Borrowing some results in the theory of orbit counting, we study the asymptotic of the growth of degrees of elements in  the orbit of a curve on an algebraic surface with respect to a geometrically finite group of its automorphisms. 

\end{abstract}

\maketitle

\section{Introduction} 
Let $\Gamma$ be a discrete subgroup of isometries of a hyperbolic  space $\bbH^n$ and $\rmO_\Gamma(x_0)$ be its orbit. Consider a family of compact subsets $\calB_T$ of $\bbH^n$ whose volume  tends to infinity as $T\to \infty$. The problem of finding the asymptotic of 
$\#\rmO_\Gamma(x_0)\cap \calB_T$ is a fundamental problem in  harmonic analysis, number theory, ergodic theory and 
geometry. In this paper we discuss an  application of this problem to algebraic geometry. 

Let $X$ be a smooth  projective algebraic surface and let $\Num(X)$ be the group of divisor classes on $X$ modulo numerical equivalence. For any divisor class $D$ we denote by $[D]$ its image in $\Num(X)$. The intersection form on divisor classes 
defines a non-degenerate symmetric bilinear form  on $\Num(X)$ of signature $(1,n)$. Let $\Gamma'$ be a group of automorphisms of $X$ such that its image $\Gamma$ in the orthogonal group $\Or(\Num(X))$ is an infinite group.  This implies that $X$ is birationally isomorphic to an abelian surface, a K3 surface, an Enriques surface or the projective plane. 

For any divisor class $D$ we denote by $\rmO_\Gamma(D)$ the $\Gamma$-orbit of $[D]$.   Fix an ample divisor class $H$ and an effective divisor class $C$. For any positive 
real number  $T$, let 
{\small $$N_{T}(H,C) = \#\{[C']\in \rmO_\Gamma(C):H\cdot C' \le T\} = \#\{H'\in \rmO_\Gamma(H):H'\cdot C \le T\}.$$}
We are interested in the asymptotic of this function when $T$ goes to infinity. As far  as I know, this problem was first considered by Arthur Baragar in his two papers \cite{Baragar1}, \cite{Baragar2}.\footnote{I thank Serge Cantat for these references.}

In the present paper we explain how does this problem relate to a general orbital counting problem in the theory of discrete subgroups of Lie groups. More precisely, we consider the group $G = \SO(1,n)_0$ realized as the group of orientation preserving isometries of the 
hyperbolic space $\bbH^n \subset \bbP(V)$ associated with the real inner product vector space $V = \Num(X)_\bbR$ of 
signature $(1,n)$. We will represent points of $\bbH^n$ (resp. points in its boundary $\partial\bbH^n$, resp. the points in 
$\bbP^n\setminus \overline{\bbH^n}$) by  vectors $v\in V$ with $(v,v) = 1, (v,h_0) > 0$ (resp. positive rays of vectors with $(v,v) = 0, (v,h_0) > 0$, resp. vectors with $(v,v) = -1,(v,h_0) > 0$), where $h_0$ is a fixed vector with $(h_0,h_0) = 1$.  Let $\calB_T(e)$ be the set of points $x$ in $\bbH^n$ whose hyperbolic distance from the point $e\in \bbH^n$ (resp. from a fixed horosphere with center at $e$ if $e\in \partial\bbH^n$, resp. from the  orthogonal hyperplane $H_e$  with normal vector 
$e$ if $e\not\in \overline{\bbH^n}$) is less than or equal to $T$.  We use the following result (see \cite{OM}) which represents the state of the art in the study of orbital counting.

\begin{theorem}\label{intro:thm1} Assume that $\Gamma$ is a non-elementary geometrically finite discrete subgroup of 
orientation preserving isometries of 
$\bbH^n$. Let $\delta_\Gamma$ be the Hausdorff dimension of the limit set $\Lambda(\Gamma)$ of $\Gamma$.  If $(e,e)\le 0$, we additionally assume that the orbit $\rmO_\Gamma([e])$ is a discrete set and $\delta_\Gamma > 1$ if $(e,e) < 0$.  Then there exists a positive constant $c_{\Gamma,h_0,e}$ depending only on $\Gamma,[h]$ and $e$ such that 
$$\underset{T\to \infty}{\lim}\frac{\#\rmO_{\Gamma}(h_0)\cap \calB_T}{\exp(T)^{\delta_\Gamma}} = c_{\Gamma,h_0,e}.$$
\end{theorem}

We apply  this theorem to our situation to obtain the following theorem.

\begin{theorem}\label{intro:thm2} Let $X$ be an abelian surface, or a K3 surface, or an Enriques surface, or a rational 
surface. Let $\Gamma'$ be a group of automorphisms of $X$ such that its image in $\Or(\Num(X))$ is a non-elementary geometrically finite discrete group. Fix an ample numerical divisor class 
$H\in \Num(X)$ and an effective numerical divisor class  $C\in \Num(X)$. Assume that $\delta_\Gamma > 1$ if $C_0^2 < 0$. Then  there exists a positive constant $c_{\Gamma,H,C}$ depending only on $\Gamma, [H]$ and $[C]$ such that 
$$\underset{T\to \infty}{\lim}\frac{N_{T}(H,C)}{T^{\delta_\Gamma}} = c_{\Gamma,H,C}.
$$
\end{theorem}

As we see, the main ingredient of the asymptotic expression  is the Hausdorff dimension $\delta_\Gamma$ of the limit set $\Lambda(\Gamma) \subset \partial \bbH^n$ of $\Gamma$. One of the first and the most beautiful example where $\delta_\Gamma$ was computed with many decimals is the example where $\Gamma$ is an Apollonian group in $\bbH^3$   whose limit set is the closure of a countable union of circles that intersect in at most one point, an \emph{Apollonian gasket}. It was proved by D. Boyd \cite{Boyd} that the asymptotic of curvatures of circles in an Apollonian gasket is equal to $c(T)T^{\delta_\Gamma}$ for some function $c(T)$ such that $\underset{T\to \infty}{\lim}\frac{c(T)}{T} = 0$. He also showed that $1<\delta_\Gamma<2$. It is known now that  
$\delta$ is about $1.305686729...$ \cite{Thomas}. An Apollonian sphere packing is a special case of a  Boyd-Maxwell sphere packing defined by certain Coxeter groups of hyperbolic type introduced by G. Maxwell in \cite{Maxwell}.\footnote{We borrow the terminology from \cite{Chen1}}.  The Hausdorff dimension of the limit set of the Boyd-Maxwell Coxeter groups coincides with 
the sphere packing critical exponent equal to 
\beq\label{crit}
 \inf\{s:\sum_{S\in \frakP} r(S)^s < \infty\} = \sup\{s:\sum_{S\in \frakP}r(S)^s = \infty\},
 \eeq
where $r(S)$ denotes the radius of a sphere $S$ in a sphere packing $\frakP$. Although the Hausdorff dimension of the limit set of a geometrically finite group is usually difficult to compute,  the bounds for sphere packing critical exponent is easier to find. 

We give several concrete examples where the Boyd-Maxwell groups, and, in particular the Apollonian groups, are realized as automorphism groups of algebraic surfaces so that our counting problem can be solved in terms of the sphere packing critical exponent. 
Since this paper addresses a problem in algebraic geometry, we give a short  introduction to the hyperbolic geometry and the theory of sphere packings in a language which we think is more suitable for algebraic geometers. Whenever the details or references are 
omitted, they can be found in \cite{Vinberg}, \cite{Apanasov}, \cite{Ratcliffe}, or  \cite{VinbergS}. 

It is my pleasure to thank Peter Sarnak for his inspiring  talks in Ann Arbor, in April 2014, that gave rise to the present paper. The paper could never be written without the assitance of Jeffrey Lagarias who patiently answered  my numerous questions on the theory of Apollonian sphere packings, of Boris Apanasov and Serge Cantat who helped me to correct some of my illiteracy in the theory of Kleinian groups, and of Hee Oh for explaining to me the state of the art in the general orbital counting problem.  
  
I am also thankful to the organizers of the conference at Schiermonnikoog2014 for giving me an opportunity to talk about the beautiful subject of Apollonian circle packings.

\section{Hyperbolic space} 
Let $\bbR^{n,1}$ denote the Minkowski space defined by the quadratic form (the \emph{fundamental quadratic form})
$$q = -t_0^2+\sum_{i=1}^nt_i^2$$
of signature $(n,1)$.\footnote{Note that we use the signature $(n,1)$ instead of $(1,n)$ to agree with  the standard notations in 
hyperbolic geometry.}  We denote by $(v,w)$ the value of the associated symmetric form on two vectors 
$v,w\in \bbR^{n,1}$, so that $q(v) = (v,v)$. Let
$$\bbH^{n} = \{v\in \bbR^{n,1}:q(v) < 0\}/\bbR^* \subset \bbP^{n+1}(\bbR)$$ 
be the $n$-dimensional \emph{hyperbolic} (or \emph{Lobachevsky}) space, a Riemannian space of
 constant negative curvature.  Its points are lines $[v] = \bbR\cdot v$ spanned by a vector 
$v\in \bbR^{n,1}$ with $(v,v) < 0$. The set of real points of the \emph{fundamental quadric} $\sfQ:q = 0$  is called the \emph{absolute}. The union of $\bbH^{n}$ and the absolute is the closure $\overline{\bbH}^{n}$ of $\bbH^{n}$ in $\bbP^{n}(\bbR)$. 
In affine coordinates $y_i = x_i/x_0$, one represents $\bbH^{n}$ by the interior of the $n$-dimensional ball
$$B_{n} = \{(y_1,\ldots,y_{n}): \sum_{i=1}^{n}y_i^2 < 1\}.$$
This is called the \emph{Klein model} of $\bbH^n$. In this model the Riemannian metric of constant curvature $-1$ is given by
$$ds^2 = \frac{(1-|y|^2)\sum_{i=1}^ndy_i^2+(\sum_{i=1}^ny_idy_i)^2}{(1-|y|^2)^2},$$
where $|y|^2 = \sum_{i=1}^ny_i^2$. 
 One can also represent any point in $\bbH^{n}$ by a vector 
$v = (x_0,\ldots,x_{n})\in \bbR^{n,1}$ with $(v,v) = -1, x_0 > 0$. Thus, we obtain a model of 
$\bbH^{n}$ as one sheet of a two-sheeted hyperboloid $Y^n$ in $\bbR^{n,1}$
$$-t_0^2+\sum_{i=1}^{n+1}t_i^2 = -1,\  t_0 > 0.$$
This model is called the \emph{vector model} of $\bbH^n$. The hyperbolic distance $d(v,w)$ in this model between two points in $\bbH^n$ satisfies
\beq\label{dist1}
\cosh d(v,w) = -(v,w).
\eeq
Here 
and later we use that, for any two vectors $v,w$ in the vector model of $\bbH^n$, we have 
$(v,w) < 0$ (see \cite{Ratcliffe}, Theorem 3.1.1).
The maps 
$$B^n\to Y^n,\  (y_1,\ldots,y_n)\mapsto \bigl(\frac{1}{\sqrt{1-|y|^2}},\frac{y_1}{\sqrt{1-|y|^2}},\ldots,\frac{y_n}{\sqrt{1-|y|^2}}\bigr)$$
and
$$Y^n\to B_n, \ (t_0,\ldots,t_{n+1}) \mapsto \bigl(\frac{t_1}{t_0},\ldots,\frac{t_n}{t_0}\bigr)$$
are inverse to each other.

The induced metric of $\bbR^{n,1}$
$
ds^2 = -dt_0^2+\sum_{i=1}^{n}dt_i^2$,
 after the scaling transformation 
$y_i = t_i/q(t)^{1/2}$, defines a Riemannian metric on $\bbH^{n}$ of constant curvature $-1$
\beq\label{m1}
ds^2 = \frac{1}{(1-|y|^2)^2}\bigl((1-|y|^2)\sum_{i=1}^{n}dy_i^2+(\sum_{i=1}^{n}y_idy_i)^2\bigr),
\eeq
where $|y|^2 = y_1^2+\cdots+y_n^2$. Consider the section of $B_{n}$ by the the hyperplane $y_{n} = 0$. It is isomorphic to the $n-1$-dimensional ball
\beq\label{ball1}
'B_{n-1} = \{(y_1,\ldots,y_{n-1})\in \bbR^{n-1}:\sum_{i=1}^{n-1}y_i^2 < 1\}.
\eeq
The restriction of the metric \eqref{m1} to $'B_{n-1}$ coincides with the metric of $\bbH^{n-1}$. Thus we may view 
$'B_{n-1}$ as a ball model of a geodesic hypersurface of $B_n = \bbH^n$. 

Let $p = [1,0,\ldots,0,-1]\in \sfQ(\bbR)$ be the southern pole of the absolute. The projection
$\bbP^n\dasharrow \bbP^{n-1}$ from this point is given by the formula 
$[t_0,\ldots,t_n]\mapsto [t_1,\ldots,t_{n-1},t_0+t_n]$. Since the hyperplanes 
$t_0+t_n = 0, t_0 = 0$ do not intersect  $\sfQ(\bbR)$, 
$\sfQ(\bbR)\setminus \{p\}$ maps bijectively onto $\bbR^{n-1}$ with coordinates $u_i = \frac{y_i}{y_n+1}$. The rational
map  defined by the formula
\beq\label{phi1}
\Phi:\bbP^{n-1}\to \sfQ,\  [x_0,\ldots,x_{n-1}]\mapsto 
[x_0^2+|x|^2,2x_0x_1,\ldots,2x_0x_{n-1},x_0^2-|x|^2],
\eeq
where $|x|^2 = x_1^2+\cdots+x_{n-1}^2$, blows down the hyperplane $x_0 = 0$ to the point $p$ and equals the inverse of the restriction of the projection $\sfQ(\bbR)\setminus \{p\}\to \bbP^{n-1}(\bbR)$.
 In affine coordinates $u_i = x_i/x_0$ in $\bbP^{n-1}\setminus \{x_0 = 0\}$ and affine 
coordinates $(y_1,\ldots,y_n)$ in $\sfQ\setminus \{t_0= 0\}$ the map is defined by the formula
\beq\label{form1}
(u_1,\ldots,u_{n-1}) \mapsto (y_1,\ldots,y_n) = 
(\frac{2u_1}{1+|u|^2},\ldots,\frac{2u_{n-1}}{1+|u|^2},\frac{1-|u|^2}{1+|u|^2}),
\eeq
where $|u|^2 = u_1^2+\cdots+u_{n-1}^2$. Let $Q^+$ be the part of $\sfQ(\bbR)$ where $y_n > 0$ (the northern hemisphere). 
The preimage of $Q^+$ under the map $\Phi$ is the subset of $\bbR^{n-1}$ defined by the inequality $|u|^2 < 1$. It is 
a unit ball $B_{n-1}$ of dimension $n-1$. Consider the orthogonal projection $\alpha$ of $Q^+$ to the ball 
$B_{n-1}$ from \eqref{ball1}. The composition 
\beq\label{comp}
\alpha\circ\Phi:B_{n-1} \to 'B_{n-1}, \ (u_1,\ldots,u_{n-1})\mapsto 
\bigl(\frac{2u_1}{1+|u|^2},\ldots,\frac{2u_{n-1}}{1+|u|^2}\bigr)
\eeq
 with the inverse 
 $$'B_{n-1}\to B_{n-1}, \ (y_1,\ldots,y_{n-1}) \mapsto  
 \bigl(\frac{y_1}{1+\sqrt{1-|y|^2}},\ldots,\frac{y_{n-1}}{1-\sqrt{1+|y|^2}}\bigr)$$
 is a bijection which we can use to transport the metric on 
$'B_{n-1}$ to a metric on $B_{n-1}$ with curvature equal to $-1$. In coordinates $u_1,\ldots,u_{n-1},$ it is given by the formula
$$ds^2 = 4 (1-|u|^2)^{-2}\sum_{i=1}^{n-1}du_i^2.$$
This is a metric on $B_{n-1}$ of constant curvature $-1$, called the \emph{Poincar\'e metric}. The corresponding 
model of the hyperbolic space $\bbH^{n-1}$ is called a \emph{conformal} or \emph{Poincar\'e model} of $\bbH^{n-1}$.

Yet, another model of $\bbH^n$ is realized in the \emph{upper half-space}
$$\calH^n = \{(u_1,\ldots,u_n)\in \bbR^n, u_n > 0\}.$$
The isomorphism $B_n\to \calH^n$ is defined by 
$$(u_1,\ldots,u_n)  = \frac{1}{\rho^2}(2y_1,\ldots,2y_{n-1},2(y_n+1)-\rho^2),$$
where $\rho^2 = y_1^2+\cdots+y_{n-1}^2+(y_n+1)^2$. The metric on $\calH^n$ is given by 
$$ds^2 = \frac{1}{y_n^2}\sum_{i=1}^ndu_i^2.$$

A \emph{$k$-plane}  $\Pi$ in $\oh^{n}$ is the non-empty intersection of a $k$-plane in $\bbP^{n}(\bbR)$ with $\oh^{n}$. We call it \emph{proper} if its intersection with $\bbH^{n}$ is non-empty, and hence, it is a geodesic submanifold  of $\bbH^{n}$ of dimension $k$ isomorphic to $\bbH^{k}$. 

A \emph{geodesic line} corresponds to a 2-dimensional subspace $U$ of signature $(1,1)$. It has a basis $(f,g)$ that consists of isotropic vectors with $(f,g) = -1$ (called a \emph{standard basis}). Thus any geodesic line $\ell$ intersects the absolute at two different points $[f]$ and $[g]$. We can choose a parameterization $\gamma:\bbR\to l$ of $\ell$ of the form $\gamma(t) = \frac{1}{\sqrt{2}}(e^tf+e^{-t}g)$ so that $\gamma(-\infty) = [f]$ and $\gamma(+\infty) = [g]$. The hyperbolic distance $d(\gamma(0),\gamma(t))$ from the point $\gamma(0) = \frac{1}{\sqrt{2}}(f+g)$ to the point $\gamma(t)$ satisfies 
$d(\gamma(0),\gamma(t)) = -(\gamma(0),\gamma(t)) = \cosh t$. Thus $t = d(\gamma(0),\gamma(t))$ is the natural parameter of the geodesic line. The subgroup of $\SO(U)_0 \cong SU(1,1)_0$ is isomorphic to $\bbR_{> 0}$ and consists of transformation $g_t$ given in the standard basis by the matrix $\bsm e^t&0\\
0&e^{-t}\esm$. It can be extended to a group of  isometries  of $\bbR^{n,1}$ that acts identically on $U^\perp$. Comparing with the natural parameterization of $l$, we see that each transformation $g_t$ moves a point $x\in l$ to a point $g_t(x)$ such that 
$d(x,g_t(x)) = t$. For this reason, it is called the \emph{hyperbolic translate}. 


A $n-1$-plane $H$ is called a \emph{hyperplane}. We write it in the form 
$$H_{\frake} = \{v\in \bbR^{n,1}\setminus \{0\}:(v,\frake) = 0, (v,v)\le 0\}/\bbR^*.$$
 Since the signature of $q$ is equal to $(n,1)$,  $\bbH^{n}\cap H_\frake = \emptyset$  if and only if $(\frake,\frake) \ge 0$. If $(\frake,\frake) = 0$, the intersection with $\oh^{n+1}$ consists of one point $[\frake]$. We assume that $(\frake,\frake) \ge  0$ and $(\frake,\frake) = 1$ if $(\frake,\frake) > 0$. These properties define $\frake$ uniquely, up to multiplication by a constant (equal to $\pm 1$ if $(\frake,\frake) = 1$). A choice of one of the two rays of the line $\bbR \frake$ is called an \emph{orientation} of the hyperplane. The complement of $H_\frake$ in $\bbH^{n}$ consists of two connected components (half-spaces) defined by the sign of $(v,\frake)$. We denote the closures of these components by $H_\frake^{\pm}$, accordingly. Changing the orientation, interchanges the half-spaces.

Consider a hyperplane  $H_{\frake'}$ in $'B_{n-1}$. It is given by an equation 
$$a_0t_0+\sum_{i=1}^{n-1}a_it_{i} = 0$$
 for some $(a_0,\ldots,a_{n-1})\in \bbR^{n-1,1}$ with norm $1$.

 Under 
the map \eqref{comp}, the pre-image of $H_{\frake'}$ is 
equal to the intersection $S(\frake')\cap \bar{B}_{n-1}$, where 
$$S(\frake'): = -a_0\sum_{i=1}^{n-1}u_i^2+2\sum_{i=1}^{n-1}a_iu_i-a_0 = 0.
$$
If $k = a_0 \ne 0$, we can rewrite this equation in the form
$$k^2\sum_{i=1}^{n-1} (u_i-a_i/k)^2 = \sum_{i=1}^{n-1} a_i^2-a_0^2 =  1.$$
This is a sphere of radius $r = \frac{1}{k}$ with the center $c = (a_1/k,\ldots,a_{n-1}/k)$. It is immediately checked that 
$S(\frake')$ intersects the boundary of $B_{n-1}$ orthogonally (i.e the radius-vectors are perpendicular at each intersection point in the 
Euclidean inner product). The intersection of $S(\frake')$ with $B_{n-1}$ is a hyperbolic codimension 1 subspace of the Poincar\'e model of $\bbH^{n-1}$. 

Assume $k =  0$,  the equation of $S(\frake')$ becomes 
$\sum_{i=1}^{n-1}a_iu_i = 0.$
 The hyperplane $S(\frake')$ should be 
viewed as a sphere of radius $r = \infty$ with the center at the point at infinity corresponding to $p$ in the compactification 
$\bbR^n\cup \{\infty\} \cong \sfQ(\bbR)$. For any $k\ge 0$, the intersection  $S(\frake')\cap B_{n-1}$ is a geodesic hypersurface in 
$\bbH^{n-1}$. It contains the center of the ball if and only if $a_0 = 0$. 

\begin{example} A one-dimensional hyperbolic space can be modeled by an open interval $(-1,1)$. The Poincar\'e and the 
Klein metrics coincide and equal to $dy^2/(1-y^2)^2$. It is isomorphic to the Euclidean one-dimensional 
space in this case.

Assume $n = 2$. We have two models of a hyperbolic 
2-dimensional space $\bbH^2$ both realized in the 
unit disk $B_2$. The the Klein model is defined by the metric 
$$ds^2 = \frac{y_1dy_2^2+y_2dy_1^2}{(1-y_1^2-y_2^2)^2}.$$
The Poincar\'e model is defined by the metric
$$ds^2 = \frac{dy_1^2+dy_2^2}{(1-y_1^2-y_2^2)^2}.$$
In the Klein model, a geodesic line in $\bbH^2$ is a line joining two distinct points on the absolute. In 
the Poincar\'e model, a geodesic line is a circle arc intersecting the absolute orthogonally at two distinct 
points or a line passing through the center of the disk.
\end{example}

The group $\Iso(\bbH^n)$ of isometries of $\bbH^n$ is the index 2 subgroup $\Or(n,1)'$ of the orthogonal group of $\bbR^{n.1}$ that preserves each sheet of the hyperboloid representing the vector model of $\bbH^n$.  When $n$ is even, we have $\Or(n,1)' = \SO(n,1)$. It is also isomorphic to the group $\PO(n,1)$ of  projective transformations of $\bbP^n$ that leaves invariant the quadric $\sfQ = \partial \bbH^n$. It consists of two connected components, the connected component of the identity $\PO(n,1)_0$ is the group of isometries that preserve the orientation. It is isomorphic to the subgroup of $\SO(n,1)_0$ of elements of spinor norm 1.  
 
 Via  the rational map \eqref{phi1}, the group $\Iso(\bbH^n)$ becomes isomorphic to the  subgroup  of the real Cremona group $\Cr_{\bbR}(n-1)$ of $\bbP_\bbR^{n-1}$.  It is generated by 
affine orthogonal transformations of $\bbP_\bbR^{n-1}$ and the inversion birational quadratic transformation
\beq\label{inv1}
[x_0,\ldots,x_{n-1}]\mapsto [\sum_{i=1}^{n-1}x_i^2,x_0x_1,\ldots,x_0x_{n-1}].
\eeq
This group is known classically as the \emph{Inversive group}  in dimension $n-1$. We can also identify $\sfQ$ with a one-point compactification $\hat{E}^{n-1} = \bbR^{n-1}\cup \{\infty\}$ of the Euclidean space of dimension $n-1$. The point $\infty$ corresponds to the point on $\sfQ$ from which we project to $\bbP^{n-1}$. In this model the Inversive group is known as the M\"obius group $\Mob(n-1)$. If we identify, in the usual way, the 2-dimensional sphere with the complex projective line $\bbP^1(\bbC)$, we obtain that 
the group $\Mob(2)$ is isomorphic to $\textrm{PSL}_2(\bbC)$ and consists of M\"obius transformations 
$z\mapsto \frac{az+b}{cz+d}$. The inversion transformation \eqref{inv1} is conjugate to the transformation 
$z\mapsto -1/z$.

It is known that the orthogonal group $\Or(n,1)$ is generated by \emph{reflections} 
\beq\label{ref1}
s_\alpha:v\mapsto v-2\frac{(v,\alpha)}{(\alpha,\alpha)}\alpha,
\eeq 
where $(\alpha,\alpha) \ne 0$. If $(\alpha,\alpha) < 0$, then $s_\alpha$ has only one fixed point in $\oh^n$, the point 
$[\alpha]\in \bbH^n$. If $(\alpha,\alpha) > 0$, then the fixed locus of $s_\alpha$ is a hyperplane in $\oh^n$. The pre-image of  its intersection with the absolute 
$\partial\bbH^n$ is the sphere $S(\alpha/(\alpha,\alpha)^{1/2})$. The M\"obius transformation corresponding to $s_\alpha$ is the inversion transformation that fixes the sphere. 

An isometry $\gamma$ of a hyperbolic space $\bbH^n$ is of the following three possible types: 
 \begin{itemize}
 \item hyperbolic  if $\gamma$ has no fixed points in $\bbH^n$ but has two fixed points on the boundary;   
 \item parabolic if $\gamma$ has no fixed points in $\bbH^n$ but has one fixed point on the boundary; 
 \item elliptic if $\gamma$ has a fixed point in $\bbH^n$.
 \end{itemize}

An element is elliptic if and only if it is of finite order. In the $\hat{E}^{n-1}$ model of the absolute,  a parabolic element corresponds to a translation in the affine orthogonal group $\textrm{AO}(n-1)$.

A discrete subgroup $\Gamma$ of $\Iso(\bbH^n)$ is called a \emph{Kleinian group}. It acts discontinuously on $\bbH^n$ and the  extension of this action to the absolute acts discontinuously on the set $\Omega(\Gamma) = \partial\bbH^n\setminus \Lambda(\Gamma)$, where 
$\Lambda(\Gamma)$ is the set of limit points of $\Gamma$, i.e. points on the boundary that belongs to the closure of the 
$\Gamma$-orbit of a point $x$ in $\bbH^n$ (does not matter which one). The set $\Omega(\Gamma)$ is called the 
\emph{discontinuity set} of $\Gamma$. The set $\bbH^n\cup \Omega(\Gamma)$ is the largest open subset of $\oh^n$ on which $\Gamma$ acts discontinuously. 

Following \cite{Ratcliffe}, \S 12.2, we call a Kleinian group  to be of the \emph{first kind} (resp. of the \emph{second type}) if its discontinuity set is empty (resp. not empty). 

A Kleinian group $\Gamma$ is called \emph{elementary} if it has a finite orbit in $\oh^n$. It is characterized by the property that it contains a free abelian subgroup of finite index. 

The limit set $\Lambda(\Gamma)$ of an elementary Kleinian group is either empty or consists of at most two points, the fixed points of a parabolic or a hyperbolic element in $\Gamma$. Otherwise $\Lambda(\Gamma)$ is a perfect set, i.e. has no isolated points. A fixed point of a hyperbolic or a parabolic transformation  is a limit point. The limit set of a non-elementary Kleinian group is equal to the closure of the set of fixed points of parabolic or hyperbolic elements \cite{Ratcliffe}, Theorem 12.2.2.

A Kleinian group  in dimension 2 is called a \emph{Fuchsian group}. Usually one assumes that $\Gamma$ is contained in the connected component of $\Iso(2)$, i.e. it preserves the orientation. The discontinuity set  $\Omega(\Gamma)$ is either empty, and then $\Gamma$ is called of the first kind, or consists of the union of open intervals, then $\gamma$ is of the second kind. The quotient $X = \bbH^2\cup \Omega(\Gamma)/\Gamma$ is a Riemann surface. If $\Gamma$ is of the first kind, one can compactify 
$X$ by a finite set of  orbits of cuspidal limit points. A Fuchsian group with $\bbH^2/\Gamma$ of finite volume is of the first kind. If $\Gamma$ is of the second kind, then $X$ is a Riemann surface with possibly a non-empty boundary. 

\section{Geometrically finite Kleinian groups}
A non-empty subset $A$ of a metric space is called \emph{convex} if any geodesic connecting two of its points is contained in $A$. Its closure is called a \emph{closed convex subset}. A closed convex subset $P$ of $\bbH^n$ is equal to the 
intersection of 
open half-spaces $H_{\frake_i}, i\in I$. It is assumed that none of the half-spaces $H_{\frake_i}$  contains the intersection of other half-spaces. A maximal non-empty closed convex subset of the boundary of $P$ is called a \emph{side}. Each side 
is equal to the intersection of $P$ with one of the bounding hyperplanes $H_{\frake_i}$. The boundary of $P$ is equal to the union of 
sides and two different sides intersect only along their boundaries.

A \emph{convex polyhedron} in $\bbH^n$ is a convex closed subset $P$ with finitely many sides. It is bounded by a finite set of 
hyperplanes $H_i$.

 We choose the vectors $\frake_i$ as in the previous section  so that 
\beq\label{pol1}
P = \bigcap_{i=1}^N H_{\frake_i}^-.
\eeq

The \emph{dihedral angle} $\phi(H_{\frake_i},H_{\frake_j})$ between two proper bounding hyperplanes is defined by the formula
$$\cos \phi(H_{\frake_i},H_{\frake_j}): = -(\frake_i,\frake_j).$$
If $|(\frake_i,\frake_j)| > 1$, the angle is not defined, we say that the hyperplanes are \emph{divergent}. In this case the distance between $H_\frake$ and $H_{\frake'}$ can be found from the formula
\beq\label{dist2}
\cosh d(H_\frake,H_{\frake'}) = |(\frake,\frake')|.
\eeq

For any subset $J$ of $\{1,\ldots,N\}$ of cardinality $n-k$, such the  intersection of 
$H_{\frake_i}, i\in J$ is a $k$-plane $\Pi$ in $\bbH^n$, the intersection $P\cap \Pi$ is a polyhedron in 
$\Pi$. It is called a \emph{$k$-face} of $P$. A $(n-1)$-face of $P$ is a side of $P$. 

The matrix 
$$G(P) = (g_{ij}), \quad g_{ij} = (\frake_i,\frake_j),$$
is called the \emph{Gram matrix} of $P$. There is a natural bijection between the set of its $k$-dimensional proper faces and 
positive definite principal submatrices of $G(P)$ of size $n-k$. The improper vertices of $P$ correspond to positive semi-definite principal submatrices of size $n$. 

 Recall that a closed subset $D$ of  a metric space $X$ is called a \emph{fundamental domain} for a group $\Gamma$ of isometries of $X$ if 
 \begin{itemize} 
 \item[(i)] the interior $D^o$ of $D$ is an open  set;
\item[(ii)] $\gamma(D^o)\cap D^o = \emptyset$, for any $\gamma\in \Gamma\setminus \{1\}$;
\item[(iii)] the set of subsets of the form $\gamma(D)$ is locally finite\footnote{This means that every point $x\in \bbH^n$ is contained in a finite set of subsets $\gamma(D)$};
\item[(iv)] $X = \cup_{\gamma\in \Gamma}\gamma(D)$;
\end{itemize}

A group $\Gamma$ admits a  fundamental domain if and only if it is a 
discrete subgroup of the group of isometries of $X$.  For example, one can choose $D$ to be a \emph{Dirichlet domain} 
\beq
D(x_0) = \{x\in X:d(x,x_0)\le d(\gamma(x),x_0), \ \textrm{for any $\gamma\in \Gamma$}\},
\eeq
where $x_0$ is fixed point in $X$ and $d(x,y)$ denotes the distance between two points. Assume $X = \bbH^n$. For any $\gamma\in \Gamma\setminus \Gamma_{x_0}$, let $H_\gamma$ be the hyperplane of points $x$ such that 
$d(x_0,x) = \rho(x,\gamma(x_0))$. Then $D(x_0) = \cap_{\gamma\in \Gamma\setminus \Gamma_{x_0}}H_\gamma^-$ and, for any  
$\gamma\not \in \Gamma_{x_0}$, $S = \gamma(D(x_0))\cap D(x_0)\subset H_\gamma$ is a side in $\partial D(x_0)$. Each side of $D(x_0)$ is obtained in this way for a unique $\gamma$.

A fundamental domain $D$ for a Kleinian group $\Gamma$ in $\bbH^n$ is called \emph{polyhedral} if its boundary $\partial D = D\setminus D^o$ is contained in the union of a locally finite set of hyperplanes $H_i$ and each side $S$ of the boundary is equal to $D\cap \gamma(D)$ for a unique $\sigma_S\in \Gamma$. A Dirichlet domain is  a convex polyhedral fundamental domain.

A choice of a polyhedral fundamental domain  allows one to find a presentation of $\Gamma$ in terms of generators and relations. The set of generators is the set of elements $\gamma_S\in \Gamma$, where $S$ is a side of $D$.  A relation 
$\gamma_{S_t}\circ \cdots \circ \gamma_{S_1} = 1$ between the generators corresponds to a cycle 
$$D_0 = D, D_1 = \gamma_{S_1}(D_0), \ldots, D_t = \gamma_{S_t}\circ \cdots \circ \gamma_{S_1}(D_0) = D_0.$$

Among various equivalent definitions of a geometrically finite Kleinian group we choose the following one 
(see \cite{Apanasov}, Chapter 4, \S 1): A Kleinian group $\Gamma$ is called \emph{geometrically finite} if it admits a  polyhedral fundamental domain with finitely many sides.
\footnote{Other equivalent definition is given in terms of the convex core of $\Gamma$, the minimal convex subset of $\bbH^n$ that contains all geodesics connecting any two points in $\Lambda(\Gamma)$.}   It follows from above that such a group is finitely generated and finitely presented. The converse is true only in dimension $n\le 2$. In dimensions $n\le 3$, one can show that $\Gamma$ is geometrically finite if and only if there  exists (equivalently any) a Dirichlet fundamental domain with finitely many sides (loc. cit., Theorem 4.4).   On the other hand, for $n > 3$ there are examples of geometrically finite groups all whose Dirichlet domains have infinitely many sides (loc. cit. Theorem 4.5).

\begin{example}\label{coxeter} A convex polyhedron $P$ in $\bbH^{n}$ is called a \emph{Coxeter polyhedron} if each pair of its hyperplanes $H_{\frake_i},H_{\frake_j}$ is either divergent or  the dihedral angle between them is equal to $\pi/m_{ij}$, where $m_{ij}\in \bbZ$ or $\infty$. The \emph{Coxeter diagram} of a Coxeter polyhedron $P$ is a labeled graph whose pairs of vertices corresponding to divergent hyperplanes are joined by a dotted edge, two vertices corresponding to 
hyperplanes forming the zero angle are joined by a solid edge or by putting $\infty$ over the edge, and two vertices corresponding to the 
hyperplanes with the dihedral angle $\pi/m_{ij}\ne 0,\pi/2$ are joined by an edge with the label $m_{ij}-2$, the label is 
dropped if $m_{ij} = 3$.

A Coxeter polyhedron is a fundamental polyhedron of the reflection Kleinian group $\Gamma_P$ generated by 
the reflections $s_{\frake_i}$. For each pair of reflections $s_{\frake_i},s_{\frake_j}$ with $m_{ij}\ne \infty$, the  product $s_{\frake_i}\circ s_{\frake_j}$ is a rotation in angle $2\pi/m_{ij}$ around the codimension 2 face $H_{\frake_i}\cap H_{\frake_j}$. It follows that $(s_i\circ s_j)^{m_{ij}}$ is one of the basic relations. In fact, these relations are basic relations. An abstract 
Coxeter group is a pair $(\Gamma,S)$ consisting of a group $\Gamma$ and a set $S$ of  its generators $s_i$ such that 
$(s_is_j)^{m_{ij}} = 1$ are the basic relations, where $(m_{ij})$ is a matrix with diagonal elements equal to 1 and  off-diagonal elements equal 
to  integers $\ge 2$ or the infinity symbol $\infty$ if the order of the product is infinite. A realization of a Coxeter group $(G,S)$ in the form of $\Gamma_P$, where $P$ is a Coxeter polyhedron in a hyperbolic space $\bbH^n$ is called a hyperbolic Coxeter 
group. It follows that the cardinality of the set $S$ of generators must be equal to $n+1$. Compact Coxeter polytopes (resp. of finite volume)  are described by Lanner (resp. quasi-Lanner) diagrams which can be found for example in \cite{VinbergS}, Chapter 5, Table 4.

Let $\Gamma$ be any discrete group and $\Gamma_r$ be its subgroup generated by reflections. It is a Coxeter group with a fundamental Coxeter polyhedron $P$. Then $\Gamma_r$ is a normal subgroup of $\Gamma$ and
$$\Gamma = \Gamma_r\rtimes A(P),$$
where $A(P)$ is the subgroup of $\Gamma$ which leaves $P$ invariant. 
\end{example}

\section{Sphere packings}
Recall that  a real \emph{n-sphere} $S$ in the projective space $\mathbb{P}_\bbR^n$ over reals is the set of real points $Q(\bbR)$ of a quadric $Q$ defined over $\bbR$ that contains a given nonsingular quadric 
 $Q_0$ with $Q_0(\bbR) = \emptyset$  in a fixed hyperplane $H_\infty$. We fix the projective coordinates to assume that  
 $H_\infty: x_0 = 0$ and $Q_0: x_0=x_1^2+\cdots+x_n^2= 0$. A $1$-sphere is called a \emph{circle}, it is a conic passing through the points $[0,1,\pm i]$ at infinity. We do not exclude the case $n = 1$ where $H_\infty$ is the point $\infty = [0,1]$ and $Q_0 = \emptyset$. A $0$-dimensional sphere is a set of two real points in $\bbP_\bbR^1$ that may coincide.

Let 
$$\bbP_\bbR^n \dasharrow \bbP_\bbR^{n+1}$$ 
be the rational map  defined by the formula
$$[x_0,\ldots,x_n] \mapsto [2x_0^2,-2x_0x_1,\ldots,-2x_0x_n,\sum_{i=1}^nx_i^2].$$
Its image is the quadric $Q$ given by the equation
\beq\label{fund2}
q = 2t_0t_{n+1}+t_1^2+\cdots+t_{n}^2 = 0.
\eeq
This map is given by a choice of a basis in the linear system of quadrics containing the quadric $Q_0$. A different choice of 
a basis leads to the map \eqref{phi1} whose image is a quadric with equation $-t_0^2+\sum_{i=1}^{n+1}t_i^2 = 0$. Let $\frake = (a_0,\ldots,a_{n+1})\in \bbP^{n+1,1}$ with $(\frake,\frake) \ge 0$. We assume that $(\frake,\frake) = 1$ if $(\frake,\frake) > 0$. The pre-image of a hyperplane 
$$H_\frake: a_{0}t_{n+1}+a_{n+1}t_{0}+\sum_{i=1}^{n+1}a_it_i = 0$$
in $\bbP_\bbR^{n+1}$ is a quadric in 
$\bbP_\bbR^n$ defined by the equation 
$$a_{0}\sum_{i=1}^nx_i^2-2a_{n+1}x_0^2-2x_0\sum_{i=1}^{n}a_ix_i = 0.$$
If $a_0\ne 0$, we can rewrite this equation in the form
\beq\label{sphere1}
a_0^2\sum_{i=1}^n(\frac{x_i}{x_0}- \frac{a_i}{a_{0}})^2= 2a_0a_{n+1}+\sum_{i=1}^na_i^2 = 1.
\eeq
So, we can identify its real points with a $n$-dimensional sphere $S(\frake)$ in the Euclidean space $\bbP^{n}(\bbR)\setminus \{x_0 = 0\}$ of radius square
$r^2 = 1/a_0^2$
and the center $c =  [\frac{a_1}{a_0},\ldots,\frac{a_{n}}{a_0}]$. It is natural to call $|a_0|$ the \emph{curvature} of the sphere. If $a_0 = 0$, we get the union of two hyperplanes $H_\infty$ and $S:=\sum_{i=1}^na_ix_i+a_{n+1}x_0= 0$. In this case we set $r := \infty$.
Note that if we take $\frake = (a_0,\ldots,a_{n+1})$ in \eqref{sphere1} with $(\frake,\frake) = 0, a_0\ne 0$, we obtain a quadratic cone
with the vertex at $c$. Its set of real points consists only of the vertex.

One introduces the oriented curvature equal to $a_0$. We agree that the positive curvature corresponds to the interior of the sphere $S(\frake)$, i.e. an open ball $B(\frake)$ of radius $r$.  The negative curvature corresponds to the open exterior of the sphere, we also call it the \emph{ball} corresponding to the oriented sphere $S$ and continue to denote it $B(\frake)$. It can be considered as a ball in the extended Euclidean space $\hat{E}^{n}$.

 A \emph{sphere packing} is an infinite  set $\frakP = (S_i)_{i\in I}$ of oriented $n$-spheres such that any two of them are either 
 disjoint or touch each other. We say that a sphere packing is \emph{strict} if, additionally, no two open balls $B_i$  intersect. An example of a non-strict sphere packing is an infinite set of nested spheres.  
 We assume also that the set $\frakP$ is \emph{locally finite} in the sense that, for any $t > 0$, there exists only finitely many spheres of curvature at most $t$ in any fixed bounded region of the space.  
The condition that two spheres $S_i$ and $S_j$ are disjoint or touch each other is easily expressed in terms of linear algebra. We have the following.

\begin{lemma}\label{touch} Let $S(v)$ and $S(w)$ be two $n$-dimensional spheres corresponding to hyperplanes $H_v$ and $H_w$. Then their interiors do not intersect if and only if 
$$(v,w)\le -1.$$
The equality takes place if and only if the spheres are tangent to each other, and hence intersect at one real point.
\end{lemma}

\begin{proof} Consider the pencil of spheres  $\lambda S(v)+\mu S(w) = H(\lambda v+\mu w)$ (here, for brevity of notation,  we do not normalize the normal vector defining the hyperplane $\lambda H(v)+\mu H(\mu)$). It is clear that  
$B(v)\cap B(w) = \emptyset$ if and only if $S(v) = S(w) = \emptyset$ or the pencil contains a cone with its vertex equal to the tangency points of the spheres. A point of intersection of two spheres is defined by a vector $x$ such that 
$(x,v) = (x,w) = (x,x) = 0$. Such $x$ exists if and only if the plane $\la v,w\ra$ spanned by $v$ and $w$ is positive or semi-positive definite. This happens if and only if $(v,v)(w,w)-(v,w)^2 \ge 0$. The equality takes place if and only $x\in \la v,w\ra$. This is equivalent to that one of the spheres in the pencil is a degenerate sphere, i.e. a cone with its vertex equal to the tangency point of the two spheres. Thus we proved that $|(v,w)|\ge 1$. As we noted before, $(v,w)$ must be negative. 
\end{proof}

 Let $\Gamma$ be a geometrically finite Kleinian group in $\bbH^{n+1}$ that is identified with a subgroup of the M\"obius group $\Mob(n)$. We say that a sphere packing is a \emph{sphere $\Gamma$-packing} if 
 it is invariant with respect to the action of $\Gamma$ on spheres and has only finitely many orbits, say $N$, on the set. We also say that a sphere $\Gamma$-packing is \emph{clustered} if one can choose representatives $S_1,\ldots,S_N$ of the orbits such  that $\frakP$ is equal to the union of the sets $\{\gamma(S_1),\ldots,\gamma(S_N)\}, \gamma\in \Gamma$. Each such set is called a \emph{cluster} of the sphere packing. We will see later than Apollonian sphere packings and, their generalizations, Boyd-Maxwell sphere packings, are examples of clustered sphere $\Gamma$-packings.

Let $\Gamma_i$ be the stabilizer subgroup of some sphere $S_i$ in a sphere $\Gamma$-packing. The sphere corresponds to a 
geodesic hyperplane $H_i$ in $\bbH^{n+1}$, hence the ball $B_i$ is the Klein model of $\bbH^{n}$ on which 
$\Gamma_i$ acts as  a geometrically finite Kleinian group. Its limit points are on $S_i$. We know that they belong to the closure of the set of 
fixed points of non-elliptic elements in $\Gamma_i$, hence they are limit point of the whole $\Gamma$. 

Let us introduce the following  assumptions.
\begin{itemize}
\item[(A1)]  for any sphere $S_i$, 
 $\Lambda(\Gamma_i) = S_i.$ 
\item[(A2)] $\bigcup_{i\in I}\bar{B}_i = \hat{E}^n.$
\end{itemize}

Assumption (A1) implies that 
\beq\label{ass1}
 \overline{\bigcup_{i\in I}S_i} \subset \Lambda(\Gamma).
\eeq  
Assumption (A2) means that the packing is \emph{maximal}, i.e. one cannot find a sphere whose interior is disjoint from the interiors of all other spheres in the packing. 

If the packing is strict,  assumptions (A1) and (A1) imply  the equality in \eqref{ass1}. Indeed, suppose $x$ is a limit point of $\Gamma$ that does not lie in the closure of the union of the spheres.  By (A2), it belongs to 
the closure of some open ball $B_i$. Suppose $x$ belongs to the open ball $B_i$. We mentioned before that the limit set is the closure of the set of fixed points of 
non-elliptic elements  of  $\Gamma$.  Thus we can find some open neighborhood of $x$ inside $B_i$ that contains a fixed point of a non-elliptic element $\gamma$ in $\Gamma$. The stabilizer  subgroup $\Gamma_i$ of $B_i$ in  $\Gamma$ is a Kleinian group acting on $B_i$. It cannot contain a non-elliptic element with a fixed point in the ball. This shows that $\gamma\not\in \Gamma_i$ and hence $\gamma(B_i)\subsetneq B_i$. Thus the set $\{\gamma^n(B_i),n\in \bbZ\}$  consists of nested spheres. This is excluded by the definition of a strict sphere packing.

In the  next two sections we give examples of sphere $\Gamma$-packings satisfying  assumptions (A1) and (A2).

\section{Boyd-Maxwell sphere packings}
Let $P$ be a Coxeter polytope in $\bbH^{n+1}$ with the Gram matrix $G(P)$. Let $\frake_1,\ldots,\frake_{n+2}$ be the normal vectors of its bounding hyperplanes. We can take the set $(\frake_1,\ldots,\frake_{n+2})$ to be  a basis of the space 
$\bbR^{n+1,1}$ with the fundamental quadric \eqref{fund2}. 

Let $\omega_j$ be a vector in $\bbR^{n,1}$ uniquely  determined by the condition
\beq
(\omega_j,\frake_i) = \delta_{ij}.
\eeq
We have 
\beq\label{ff}
\omega_j = \sum_{i=1}^{n}g^{ij}\frake_i,\eeq
where 
$$G(P)^{-1} = (g^{ij}) = ((\omega_i,\omega_j)).$$
We call $\omega_i$ \emph{real} if  $(\omega_i,\omega_i) = g^{ii} > 0.$ In this case we can normalize it to set
$$\bar{\omega}_i:= \omega_i/\sqrt{g^{ii}}.$$
  Let $J = \{j_1<\cdots < j_r\}$ be the subset of $\{1,\ldots,n+2\}$ such that 
$\bar{\omega}_j$ is real if and only if $j\in J$. Consider the union
\beq\label{cluster}
\frakP(P) = \bigcup_{k = 1}^rO_{\Gamma_P}(S(\bar{\omega}_{j_k})).
\eeq

A \emph{Boyd-Maxwell sphere packing} is a sphere packing  of the form $\frakP(P)$, where $P$ is a Coxeter polytope. By definition, it is clustered with clusters 
$$(\gamma(S(\bar{\omega}_{j_1})),\ldots,\gamma(S(\bar{\omega}_{j_r}))) = (S(\gamma(\bar{\omega}_{j_1})),\ldots,
S(\gamma(\bar{\omega}_{j_r})).$$
We call the cluster $(S(\omega_{j_1}),\ldots,S(\omega_{j_r}))$ the \emph{initial cluster}. It is called \emph{non-degenerate} if all $\omega_i$ are real. This is 
equivalent to that all principal maximal minors of the matrix $G(P)$ are negative. 

Following G. Maxwell \cite{Maxwell}, we say that the Coxeter diagram is of \emph{level l} if, after deleting any $l$ of its vertices, we obtain a Coxeter diagram of Euclidean  or of parabolic type describing a Coxeter polyhedron in an Euclidean or a spherical geometry. They can be characterized by the property that all $m_{ij}\ne \infty$ (except in the case $n=1$) and the symmetric 
matrix $(m_{ij})$ is non-negative definite, definite in the Euclidean case, and having a one-dimensional radical in the parabolic case.  Coxeter diagrams of level $l = 1$ correspond to Lanner and quasi-Lanner Coxeter groups. Coxeter diagrams of level 2 have been classified by G. Maxwell (with three graphs omitted, see  \cite{Chen1}). They occur 
only in dimension $n\le 10$.  For $n\ge 4$ they are obtained from quasi-Lanner diagrams by adding one vertex. 

For example, a Coxeter 
polyhedron in $\bbH^{10}$ with Coxeter diagram of level 2

$$\xy (-20,10)*{};(-20,-15)*{};
@={(-10,0),(0,0),(10,0),(20,0),(30,0),(40,0),(50,0),(60,0),(70,0),(80,0),(10,-10)}@@{*{\bullet}};
(-10,0)*{};(80,0)*{}**\dir{-};
(10,0)*{};(10,-10)*{}**\dir{-}
\endxy
$$
defines a sphere packing with only one real $\omega_i$ corresponding to  the extreme  vertex on the right.

The following theorem is proven, under a certain assumption, in loc.cit., Theorem 3.3. The assumption had been later removed in \cite{Maxwell2}, Theorem 6.1.\footnote{I am grateful to Hao Chen for the reference.}

\begin{theorem}\label{maxwell} Let $P$ be a Coxeter polyhedron in $\bbH^n$. Then $\frakP(P)$ is a sphere 
$\Gamma_P$-packing if and only if the Coxeter diagram of $P$ is of level $\le 2$.
\end{theorem}

If $P$ is a quasi-Lanner Coxeter polytope, the Kleinian group $\Gamma_P$ has a fundamental polyhedron of finite volume. It follows that $\Gamma_P$ is a Kleinian group of the first kind (\cite{Ratcliffe}, Theorem 12.2.13). Thus any point on the absolute is its limit point. The stabilizer $\Gamma_{P,i}$ of each $S(\bar{\omega}_i)$ is the reflection group of the Coxeter group defined by the Coxeter diagram of level 1, hence the limit set of $\Gamma_{P,i}$ is equal to $S(\bar{\omega}_i)$. Thus the assumption (A1) is satisfied.

Let $\frakP(P)$ be a non-degenerate Boyd-Maxwell sphere packing and let  $\bfk = (k_1,\ldots,k_{n+2})$ be the vector of the curvatures of the spheres $S(v_i)$ in its cluster $(S(v_1),\ldots,S(v_{n+2}))$. 

\begin{theorem}\label{soddy} Let $G(P)$ be the Gram matrix of $P$. Then
 \beq\label{gdes}
 {}^t\bfk\cdot G(P)\cdot \bfk = 0.
 \eeq
 \end{theorem}
 
 \begin{proof} Let $J_{n+2}$ be the matrix of the symmetric bilinear form defined by the fundamental quadratic form $q$. It satisfies $J_{n+2} = J_{n+2}^{-1}$. Let $X$ be the matrix whose $j$th column is the vector of  coordinates of the vector  $v_j$. Recall that the first coordinate of each vector $v_j$ is equal to the curvature of the sphere $S(v_j).$ Let 
 $v_i = \gamma(\bar{\omega}_i)$ for some $\gamma\in \Gamma_P$.  We have
 $$(v_i,v_j) = (\gamma(\bar{\omega}_i),\gamma(\bar{\omega}_j)) =  (\bar{\omega}_i,\bar{\omega}_i).$$
 Thus,
$${}^tX\cdot J_{n+2}\cdot X = G(P)^{-1},$$
hence
$${}^tX^{-1}\cdot G(P)^{-1}\cdot X^{-1} = J_{n+2}.$$
Taking the inverse, we obtain
$$X\cdot G(P)\cdot {}^tX  = J_{n+2}^{-1} = J_{n+2}.$$
 The first entry $a_{11}$ of the matrix in the right-hand side is equal to zero. Hence 
$${}^t\bfk\cdot G(P)\cdot\bfk = 0,$$
\end{proof}

It is clear that  we have the following possibilities for a cluster in $\frakP(P)$.

\xy
(0,25)*{};(0,-25)*{};
(-15,0)*{};(0,0)*\cir<40pt>{};
(-15,0)*{};(2,2)*\cir<10pt>{};
(-15,0)*{};(7,7)*\cir<10pt>{}; 
(-15,0)*{};(-5,-5)*\cir<12pt>{};
(0,-20)*{(-,+,\ldots,+)}; 
(30,5)*\cir<20pt>{};
(42,2)*\cir<10pt>{};
(47,7)*\cir<10pt>{}; 
(35,-10)*\cir<10pt>{};
(35,-20)*{(+,+,\ldots,+)};
(60,15)*{};(60,-15)*{}**\dir{-};
(65,7)*\cir<10pt>{};(67,0)*\cir<10pt>{};
(65,7)*\cir<10pt>{};(67,-10)*\cir<10pt>{};
(65,-20)*{(0,+,\ldots,+)};
(85,15)*{};(85,-15)*{}**\dir{-};
(100,15)*{};(100,-15)*{}**\dir{-};
(90,7)*\cir<10pt>{};(90,0)*\cir<10pt>{};(93,-10)*\cir<10pt>{};
(95,-20)*{(0,0,+,\ldots,+)};
\endxy

Assume $n\ge 2$. Suppose $\frakP(P)$ is a non-degenerate Boyd-Maxwell sphere packing. Let 
$\bar{\omega}_1,\ldots,\bar{\omega}_{n+2}$ be defined as \eqref{ff}. Then the hyperplanes $H(\bar{\omega}_i)$ bound a convex polyhedron with no proper vertices  on the absolute. Let us denote it by $P^\perp$. Since $\frakP(P)$ is a sphere packing, no two bounding hyperplanes of $P^\perp$ have a common point in the absolute. By Lemma \ref{touch}, 
\beq
 (\bar{\omega}_i, \bar{\omega}_j)\le -1, \quad i\ne j.
\eeq
Thus any two bounding hyperplanes of $P^\perp$ either diverge or have the dihedral angle equal to $0$. In particular, we obtain that $P^\perp$ is a Coxeter polyhedron and hence its reflection group 
$\Gamma_P^\perp:= \Gamma_{P^\perp}$ is a Kleinian group. The  Coxeter diagram of $P^\perp$ has either solid or dotted edges. As a  non-weighted graph, it is a complete graph $K_{n+2}$ with $n+2$ vertices. It is clear that it is of level $l \le 2$ if and only if 
$n\le 3$. In this case we can define the \emph{dual Boyd-Maxwell sphere packing} 
$$\frakP(P)^\perp:= \frakP(P^\perp).$$

Note that the property $(\frake_i,\bar{\omega}_k) = 0$ for $i\ne k$ is interpreted in terms of spheres as the property that the spheres $S(\frake_i)$ and $S(\bar{\omega}_j)$ are orthogonal to each other (see, for example, \cite{DH}, Proposition 7.3). 
This implies that the dual sphere packing can be defined by a cluster that consists of spheres orthogonal to all 
spheres of a cluster defining $\frakP(P)$ except one.

Let $A$ be the matrix of a transformation from $\Gamma_{P^\perp}$ in a basis $(\omega_1,\ldots,\omega_{n+1})$. Then 
$${}^tA\cdot G(P)^{-1}\cdot A = G(P).$$
Taking the inverse, we obtain
\beq\label{inv}
A^{-1}\cdot G(P)\cdot {}^tA^{-1} = G(P).
\eeq
This shows that under the homomorphism $A\to {}^tA^{-1}$, the group $\Gamma_P^\perp$ becomes isomorphic to a subgroup of the orthogonal group of the quadratic form defined by the matrix $G(P)$. 

The group $\Gamma_P$ is generated by the reflections $s_{\frake_i}$ that act on the dual basis 
$(\omega_1,\ldots,\omega_{n+2})$ by the formula
\beq\label{refl1}
s_{\frake_i}(\omega_j) = \omega_j-2(\frake_i,\omega_j)\frake_i = 
\omega_j-2\delta_{ij}\frake_i = \omega_j-2\delta_{ij}\sum_{k=1}^{n+2}g_{ki}\omega_k,
\eeq
where $\delta_{ij}$ is the Kronecker symbol. This gives explicitly the action of $\Gamma_P$ on the clusters of the sphere packing and also on the set of their curvature vectors. 

The group $\Gamma_{P^\perp}$ is generated by the reflections 
$s_{\omega_i} = s_{\bar{\omega}_i}$ that act on the basis 
$(\frake_1,\ldots,\frake_{n+2})$ by the  formula
\beq\label{refl2}
s_{\omega_i}(\frake_j) = \frake_j-2(\bar{\omega}_i,\frake_j)\bar{\omega}_i =
 \frake_j-\frac{2\delta_{ij}}{g^{ii}}\omega_i = \frake_j-\frac{2\delta_{ij}}{g^{ii}}\sum_{k=1}^{n+2}g^{ki}e_k.
 \eeq

\begin{example} This is the most notorious and  widely discussed beautiful example of a Boyd-Maxwell sphere packing (see, for example, \cite{Lagarias1}, \cite{Lagarias2}, \cite{Lagarias3},\cite{Lagarias4},\cite{Sarnak}).

We start with a set of pairwise touching non-degenerate $n$-spheres $S(v_1),\ldots,S(v_{n+2})$ with positive radii that 
touch each other. By Lemma \ref{touch}, the  Gram matrix $G = ((v_i,v_j))$ has 1 at the diagonal and $-1$ off the diagonal. Thus,
$$G = \circulant(1,-1,\ldots,-1),$$
where $\circulant(\alpha_1,\ldots,\alpha_{n+2})$ denotes the circulant matrix with the first row $(\alpha_1,\ldots,\alpha_{n+2})$. 
It is easy to compute the inverse of $G$ to obtain
\beq\label{apoldual}
G^{-1} = \frac{1}{2n}\circulant(n-1,-1,\ldots,-1).
\eeq
Assume $n \ge 2$. Then  
$$\frac{2n}{n-1}G^{-1} = \circulant(1,\frac{1}{1-n},\ldots,\frac{1}{1-n}).$$
It is equal to the Gram matrix of a Coxeter polytope $P$ if and only if $n =2, 3$. The group $\Gamma_P$ is called the \emph{Apollonian group} and will be denoted by $\Ap_n$. It is a Kleinian 
group only for $n = 2,3$. 
 We call $P$ an 
\emph{Apollonian polyhedron}. Its Coxeter diagram is equal to the 
complete graph $K_4$ with solid edges if $n = 2$ and a complete graph $K_5$ with simple edges if $n = 3$. In each case 
the Coxeter diagram is of level 2. Thus, $P$ defines a Boyd-Maxwell sphere packing in dimensions $n = 2$ or $3$. 
In the case $n = 2$, we see that all circles in the cluster defined by the  vectors 
$\bar{\omega}_1,\ldots,\bar{\omega}_{n+2}$  mutually touch each other.  So, all clusters in this packing define a set of $n+2$ mutually touching circles. Any two circles in the packing are either disjoint or touch each other.

We have $G(P)^{-1} = \frac{n-1}{2n}G$, thus $P^\perp$ has the Gram matrix equal to $G$. It follows that $P^\perp$ is 
always a Coxeter polyhedron with Coxeter diagram equal to a complete graph $K_{n+2}$ with solid edges. It defines a 
sphere packing only if $n\le 3$. Clearly, it is non-degenerate, and hence by Theorem \ref{maxwell}, it is maximal. If $n = 2,3$, it is the dual to the Apollonian packing. Its clusters consist of $n+2$ mutually tangent $n$-spheres. If $n = 2$, the dual of an Apollonian circle packing is an Apollonian circle packing. Its  clusters are orthogonal to clusters of the original packing as shown on Figure 1 below. The group $\Ap_n^\perp$ is called the \emph{dual Apollonian group}.

\begin{figure}[ht]
\begin{center}
\includegraphics[scale=.8]{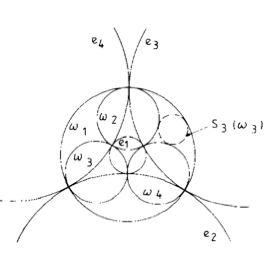}
\caption{Dual Apollonian circle clusters (from \cite{Maxwell})}
\end{center}
\end{figure}

Let $H_{\bar{\omega}_1},\ldots,H_{\bar{\omega}_{n+2}}$ be the bounding hyperplanes of  an Apollonian Coxeter polyhedron and $s_{\frake_i}$ be the corresponding reflections generating the Apollonian group $\Ap_n$. Let $S(\frake_1),\ldots, S(\frake_{n+2})$ be the cluster of the dual Apollonian sphere packing. The formula \eqref{refl1} specializes to give $s_{\frake_i}(\bar{\omega}_j) = \bar{\omega}_j$, and 
$$s_{\frake_i}(\bar{\omega}_i) = -\bar{\omega}_i+\frac{2}{n-1}\sum_{j\ne i}\bar{w}_j.$$
For example, $s_{\frake_1}$ is represented in the basis $(\bar{\omega}_1,\ldots,\bar{\omega}_{n+2})$ by the matrix
$$A_1 = \bsm-1&0&0&\ldots&0\\
\frac{2}{n-1}&1&0&\ldots&0\\
\vdots&\vdots&\vdots&\vdots&\vdots\\
\frac{2}{n-1}&0&0&\ldots&1\esm.
$$
So, the Apollonian group is generated by $n+2$ matrices $A_1,\ldots,A_{n+2}$ of this sort.  
 
Formula \eqref{refl2} shows that the dual Apollonian group $\Ap_n^\perp$ acts on the Apollonian packing by acting on spheres $S(\frake_i)$ via the reflections $s_{\bar{\omega}_i}$. We have $s_{\bar{\omega}_i}(\frake_j) = \frake_j$ if $j\ne i$ and 
$$s_{\bar{\omega}_i}(\frake_i) =  
-\frake_i+2\sum_{k=1, k\ne i}^{n+2}\frake_j$$
For example, $s_{\bar{\omega}_1}$ is represented by the matrix
$$B_1 = \bsm-1&2&0&\ldots&2\\
0&1&0&\ldots&0\\
\vdots&\vdots&\vdots&\vdots&\vdots\\
0&0&0&\ldots&1\esm.$$

The equation from  Theorem \ref{soddy} becomes
\beq\label{des1}
n(k_1^2+\cdots+k_{n+2}^2) -(k_1+\cdots+k_{n+2})^2 = 0.
\eeq
It is known as \emph{Descartes's equation} or \emph{Soddy's equation}\footnote{From a ``poem proof'' of the theorem in the case 
$n = 2$ in  ``Kiss Precise'' by Frederick Soddy published in Nature, 1930:

Four circles to the kissing come. /
The smaller are the bender. /
The bend is just the inverse of / The distance from the center. /
Though their intrigue left Euclid dumb / ThereÕs now no need for rule of thumb. /
Since zero bendÕs a dead straight line / And concave bends have minus sign, /
The sum of the squares of all four bends / Is half the square of their sum.}

Using \eqref{inv}, we see that $\Ap_n^\perp$ acts on the solutions of the equation \eqref{gdes} by the 
contragradient representations. It is generated by the matrices ${}^tA^{-1} = {}^tA_1,\ldots,{}^tA_{n+2}^{-1}= {}^tA_{n+2}$. However, if $n\ne 2$, it does not leave the Apollonian packing invariant.

Assume $n > 1$. Consider a solution $(k_1,\ldots,k_{n+2})$ of the Descartes' equation which we rewrite  in the form
\beq\label{des2}
\sum_{i=1}^{n+2} k_i^2 - \frac{2}{n-1}\sum_{1\le i\le j\le n+2}k_ik_j = 0.
\eeq
Thus, $k_{n+2}$ satisfies the quadratic equation
$$t^2-\frac{2t}{n-1}\sum_{i=1}^{n+1}k_i-\frac{2}{n-1}\sum_{1\le i<k \le n+1}^{n+2}k_i = 0.$$
It expresses the well-know fact (the \emph{Apollonian Theorem}) that, given $n+1$ spheres touching each other, there are two more spheres that touch them. Thus, starting from an Apollonian cluster $(S_1,\ldots,S_{n+2})$, we get a new Apollonian cluster $(S_1,\ldots,S_{n+1},S_{n+2}')$ such that the curvatures of $S_{n+2}$ and $S_{n+2}'$ are solution of the above quadratic equation. The curvatures of $S_{n+2}'$ is equal to 
$$k_{n+2}' = -k_{n+2}+\frac{2}{n-1}\sum_{i=1}^{n+1}k_i.$$
So the new cluster $(S_1,\ldots,S_{n+1},S_{n+2}')$ is the cluster obtained from the original one by applying the transformation $s_{\frake_{n+2}}$.

\begin{figure}[ht]
\begin{center}
\includegraphics[scale=.5]{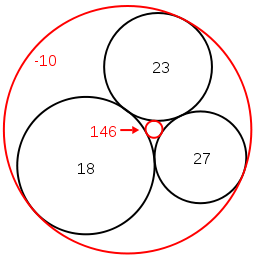}
\caption{Apollonina circle cluster with curvature vector $(-10,18,23,27)$}
\label{fig2}
\end{center}
\end{figure}

If we start with a cluster as in the picture, then all circles will be inclosed in a unique circle of 
largest radius, so our circle packing will look as in  following Figure \ref{fig3}.

\begin{figure}[ht]
\begin{center}
\includegraphics[scale=.7]{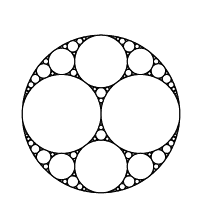}
\caption{Apollonina circle packing}
\label{fig3}
\end{center}
\end{figure}

Another possibility is when we start with a cluster of three circles touching a line. Then, the circle packing will look like in the following Figure \ref{fig4}.

\begin{figure}[ht]
\begin{center}
\quad \includegraphics[scale=.5]{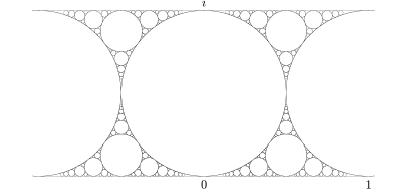}
\caption{Apollonina circle packing in a band}
\label{fig4}
\end{center}
\end{figure}

\end{example}

The following example is taken from \cite{Boyd2}.

 \begin{example}\label{boyd} We take $P$ with the Gram matrix
$$G(P) = \circulant(1,-1,0,-1).$$
It is a Coxeter polyhedron with the Coxeter diagram

\xy (-50,10)*{};(-30,-5)*{};
@={(-10,0),(0,0),(10,0),(20,0)}@@{*{\bullet}};
(-10,0)*{};(20,0)*{}**\dir{-};
(-5,2)*{\infty};(5,3)*{\infty};(15,3)*{\infty};
\endxy
The dual polyhedron $P^\perp$ has the Gram matrix $G(P)^{-1}$
$$G(P)^{-1} = \frac{1}{2}\circulant(1,-1,-2,-1).$$
\end{example}

\begin{example} Let us consider the exceptional case of Apollonian packings when $n = 1$ (this case was briefly mentioned in \cite{Lagarias4}). Here $0$-dimensional spheres are sets of two or one points in $\bbP^1(\bbR)$ that we can identify with the boundary of the unit disk model of the hyperbolic plane $\bbH^2$. We start with three $0$-dimensional spheres touching each other. This means that each pair has a common point. Let $p_{12},p_{13},p_{23}$ be the common points. We use the Poincar\'e model of $\bbH^2$. Then the corresponding hyperplanes $H(\omega_1), H(\omega_1),H(\omega_1)$ are three geodesics  joining pairs of these points. They form a \emph{ideal triangle} with vertices on the boundary.  

\begin{figure}[ht]
\begin{center}
\includegraphics[scale=.2]{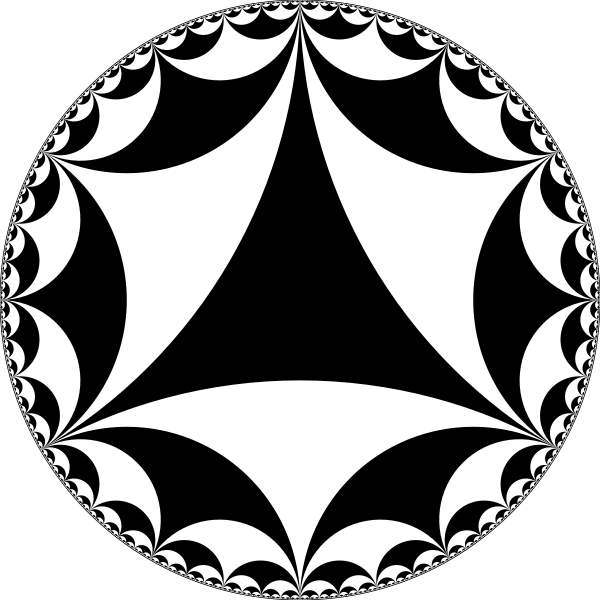}
\caption{Ideal triangle}
\end{center}
\end{figure}

The Apollonian group is not defined in this case. However, the dual Apollonian group $\Ap_1^\perp$ is defined. It is 
a Coxeter group with Coxeter diagram equal to the complete graph $K_3$ with solid edges. It is a quasi-Lanner 
Coxeter group generated by the reflections in sides of the ideal triangle. 

 The fundamental polyhedron of  $\Ap_1^\perp$ is equal to the closure of the interior of the triangle. Its  Gram matrix is equal to $\circulant(1,-1,-1)$. Since it is of finite volume, the group is a Fuchsian group of the first kind with 3 cusps. The limit set of $\Ap_1^\perp$ is the whole circle. However, the packing is not strict since the interiors of our $0$-dimensional spheres are one of the two intervals in $\bbP^1(\bbR)$ with end points equal to the sphere. Obviously we have a nested family of intervals. The set of $0$-dimensional spheres in the sphere $\Ap_1^\perp$-packing is a proper subset of 
$\Lambda(\Ap_1^\perp)$. It consists of cusp points and is countable.

The Descartes equation becomes 
\beq\label{desn=1}
k_1k_2+k_1k_3+k_2k_3 = 0.
\eeq

\end{example}

\section{Integral packings}
Let $K \subset \bbR$ be a  totally real field  of algebraic numbers of degree $d = [K:\bbQ]$ and $\frako_K$ be its ring of integers. A free $\frako_K$-submodule $L$ of rank $n+2$ of  $\bbR^{n+1,1}$ is called a $\frako_K$-lattice. Let $(\frake_1,\ldots,\frake_{n+2})$ be a basis of $L$. Assume that the entries of its Gram matrix $G = ((\frake_i,\frake_j)) = (g_{ij})$ belong to $\frako_K$
 Let $\sigma_i:K\hookrightarrow \bbR, 
i = 1,\ldots,d-1,$ be the set of non-identical embeddings of $K$ into $\bbR$. We assume additionally that the matrices 
$G^{\sigma_i} = (\sigma_i(g_{ij}))$ are positive definite. Let 
$$f = \sum_{1\le i,j\le n+2}g_{ij}t_it_j$$ 
be the quadratic form defined by the matrix $G$ and $\Or(f,\frako_K)$ be the subgroup of 
$\GL_{n+2}(\frako_K)$ of transformations 
that leave $f$ invariant. The group $\Or(f,\frako_K)$ is a discrete subgroup of $\Or(\bbR^{n+1,1})$ with a fundamental polyhedron of finite volume, compact if $f$ does not represent zero. By passing to the projective orthogonal group, we will view such groups as Kleinian subgroups of $\Iso(\bbH^{n+1})$. A subgroup of $\Iso(\bbH^{n+1})$ which is commensurable with a subgroup of the form $\PO(f,\frako_K)$  is called \emph{arithmetic}. Two groups $\PO(f',\frako_K')$ and $\PO(f,\frako_K)$ are commensurable if 
and only if $K = K'$ and $f$ is equivalent to $\lambda f'$ for some positive $\lambda\in K$. In particular, any subgroup of finite index in $\PO(f,\frako_K)$ is an arithmetic group. 

We will be interested in the special case of an \emph{integral lattice} where $K = \bbQ$. A Kleinian group 
$\Gamma$  in $\bbH^{n+1}$ is called \emph{integral} if it is commensurable with a subgroup of $\PO(f,\bbZ)$ for some integral quadratic form $f$ of signature $(n+1,1)$.
We will be dealing mostly with non-arithmetic  geometrically finite Kleinian groups. However, it follows from \cite{Benoist}, Proposition 1, that such a group is always Zariski dense in $\textrm{PSO}(n+1,1)$ or $\PO(n+1,1)$ if it acts irreducibly in $\bbH^{n+1}$. In terminology due to P. Sarnak, $\Gamma$ is a \emph{thin group}.

We will use the usual terminology of the theory of integral quadratic forms, or more abstractly, 
\emph{integral quadratic lattices}, free abelian groups $L$ of finite rank equipped with a symmetric bilinear form with values in $\bbZ$. By tensoring with $\bbR$, we embed $L$   in $L_\bbR:=L\otimes \bbR$.  We assume that the quadratic form of $L$ extends to  a quadratic form on $L_\bbR$ of signature $(n+1,1)$, this will identify $L$ with  an integral lattice in the previous sense. The set $L^\vee$ of all vectors in $L_\bbR$ such that $(x,v)\in \bbZ$ for all $x\in L$ is called the \emph{dual lattice}. It is a free abelian group spanned by the dual basis of a basis of $L$ in $L_\bbR$. The Gram matrix of this basis is the inverse of the Gram matrix of the basis of $L$. The quotient group $L^\vee/L$ is called the 
\emph{discriminant group} of $L$. Its order is equal to the minus of the 
determinant of the Gram matrix of any basis of $L$. Note that $L^\vee \subset L_\bbR$ is not, in general, an integral lattice. 
However, it becomes integral if we multiply its quadratic form by the exponent of the discriminant group of $L$.  We denote by $\Or(L)$ the orthogonal group of $L$, the subgroup of $\GL(L)$ that preserves the bilinear form.  A reflection isometry of $\bbR^{N}$ defined by a vector $v\in L$ belongs to $\Or(L)$ if and only if 
$\frac{2(v,x)}{(v,v)}\in \bbZ$ for all $x\in L$.

\begin{remark}\label{new} Let $\Gamma$ be an integral  Kleinian group of isometries of $\bbH^n$ of finite covolume, for 
example, the orthogonal group of some integral quadratic lattice $L$. It is obviously geometrically finite, however, for $n > 3$ it may contain finitely generated subgroups which are not geometrically finite. In fact, for any lattice $L$ of rank $\ge 5$  that contains a primitive sublattice defined by the quadratic form $q = -x_0^2+x_1^2+x_2^2+x_3^2$, the orthogonal group $\Or(L)$ contains finitely generated but not finitely presented subgroups (see \cite{KPV}).\footnote{I thank Boris Apanasov who informed me about this fact.}
\end{remark}

A sphere packing $\frakP = (S_i)_{i\in I}$ is called \emph{integral} if the following conditions are satisfied.
\begin{itemize}
\item[(i)] The corresponding norm one vectors $v_i$  in $\bbR^{n+1,1}$ span an integral  lattice $L$;
\item[(ii)] There exists a positive integer $\lambda$ such that $\lambda(v_i,v_j)\in \bbZ$ for all $i,j\in I$. The smallest such $c$ is called the \emph{exponent}.
\item[(iii)] The curvatures of spheres $S_i$ are integers.
\end{itemize}

 After multiplying the quadratic form of $L$, the restriction of the fundamental quadratic form to $L$, by the exponent we obtain an integral quadratic lattice $\integ(L)$. If $\Gamma$ is a Kleinian group that leaves an integral sphere packing invariant, then it  obviously leaves $\integ(L)$ invariant and becomes a subgroup of $\Or(\integ(L))$.

It follows from formula \eqref{refl1} that a Boyd-Maxwell sphere packing defined by a Coxeter polytope $P$ with Gram matrix $G(P) = (g_{ij})$ is integral if and only if the curvatures of the spheres from the initial cluster are integers and $g_{ij}\in \bbQ$.  The quadratic lattice is defined by the matrix $eG(P)$, where $e$ is the smallest 
positive integer such that $eg_{ij}\in \bbZ$.  
 
\begin{example} Let $P$ be an Apollonian polyhedron in $\bbH^{n+1}, n > 1$. It is defined by the matrix 
$\circulant(1,\frac{1}{1-n},\ldots,\frac{1}{1-n})$ whose inverse is equal to $\frac{n-1}{2n}\circulant(1,-1,\ldots,-1)$. 
Recall that it is a Coxeter polyhedron only for $n = 2,3$. The Apollonian packing is an integral Coxeter packing with 
exponent $1$. The Apollonian group $\Ap_n$ becomes a subgroup of the orthogonal group $\Or(\calAp(n))$, where $\calAp(n)$ is an integral quadratic lattice of rank $n+2$ with the quadratic form
\beq\label{aplat}
q_n = \sum_{i=1}^{n+2}t_i^2-2\sum_{1\le i< j\le n+2}t_it_j.
\eeq
Recall that an integral quadratic lattice is called \emph{even} if all its values are even integers, it is called \emph{odd} otherwise. Obviously, the lattice $\calAp(n)$ is odd. For any lattice $L$, we denote by $L^{\ev}$ its largest even sublattice. Obviously, $\Or(L)\subset \Or(L^{\ev})$. Let us find an orthogonal decomposition of the lattice $\calAp(n)^{\ev}$. 

The quadratic form \eqref{aplat} is given in a basis $\bar{\omega}_1,\ldots,\bar{\omega}_{n+2}$ of $\calAp(n)$. Consider a new basis $f_1,\ldots,f_{n+2}$  formed by the vectors 
 $\bar{\omega}_1,\bar{\omega}_1+\bar{\omega}_2,-\bar{\omega}_1-\bar{\omega}_3,\bar{\omega}_1+\bar{\omega}_2+\bar{\omega}_3-\bar{\omega}_4,\bar{\omega}_4-\bar{\omega}_5,\ldots,\bar{\omega}_{n+1}-\bar{\omega}_{n+2}$. The Gram matrix of this basis is the following matrix.
$$\bsm 1&0&0&0&0&0&0&\ldots&0&0\\
0&0&2&0&0&0&0&\ldots&0&0\\
0&2&0&0&0&0&0&\ldots&0&0\\
0&0&0&4&-2&0&0&\ldots&0&0\\
0&0&0&-2&4&-2&0&\ldots&0&0\\
\vdots&\vdots&\vdots&\vdots&\vdots&\vdots&\vdots&\vdots&\vdots&\vdots\\
0&0&0&0&0&0&0&\ldots&-2&4\esm.$$
This shows that, for any $n > 1$, 
\beq\label{mn}
\calAp(n)\cong \la 1\ra \oplus \sfU(2)\oplus \sfA_{n-1}(2).
\eeq
Here we use the standard notation from the theory of integral quadratic lattices:$\sfU$ is the integral hyperbolic plane, $\sfA_n$ is the even positive definite root lattice defined by the Cartan matrix of type $A_n$, $\la k\ra$ is the lattice of rank 1 spanned by a vector $v$ with $(v,v)= k$, and $M(t)$ denotes the quadratic lattice $M$ with its integral quadratic form multiplied by  a rational number $t$.

We have 
\beq\label{mn2}
\calAp(n)^{\ev} =  \sfA_1(2) \oplus \sfU(2)\oplus \sfA_{n-1}(2) \cong (\sfA_1 \oplus \sfU\oplus \sfA_{n-1})(2).
\eeq
The discriminant group of $\calAp(n)$ (resp. $\calAp(n)^{\ev}$) is isomorphic to 
$(\bbZ/2\bbZ)^{n}\oplus  \bbZ/2n\bbZ$ (resp. $(\bbZ/2\bbZ)^{n}\oplus  \bbZ/2n\bbZ\oplus \bbZ/4\bbZ$).

Next we consider the dual Apollonian sphere packing. Its lattice is the dual lattice of $\calAp(n)^\vee$ defined by the quadratic form from the left-hand side of \eqref{des1} multiplied by $\frac{n-1}{2n}$. 
$$q_n^\perp: = (n-1)\sum_{i=1}^{n+2}t_i^2-2\sum_{1\le i< j\le n+2}t_it_j.$$
Its exponent is equal to $2n$ and coincides with the exponent of the discriminant group of $\calAp(n)$. We have 
$$
\calAp(n)^\perp:= \integ(\calAp(n)^\vee) = \calAp(n)^\vee(2n) \cong (\la 1\ra \oplus \sfU(1/2)\oplus \sfA_{n-1}^\vee(1/2))(2n)$$
$$=\la 2n\ra \oplus \sfU(n)\oplus \sfA_{n-1}^\vee(n).
$$
The lattice $\sfA_{n-1}^\vee$ is generated by $\sfA_{n-1}$ and a vector 
$$\beta =\frac{1}{n}(\omega_1+2\omega_2+\cdots+(n-1)\omega_{n-1}),$$
where $\omega_1,\ldots,\omega_{n-1}$ is the standard basis defined by simple roots. In the basis $(\omega_2,\ldots,\omega_{n-1},\beta)$, the matrix of $\sfA_{n-1}^\vee(n)$ is equal to 
\beq\label{mat1}
\bsm 2n&-n&0&0&\ldots&0&0&0&0\\
-n&2n&-n&0&\ldots&0&0&0&0\\
0&-n&2&-n&\ldots&0&0&0&0\\
\vdots&\vdots&\vdots&\vdots&\vdots&\vdots&\vdots&\vdots&\vdots\\
0&0&0&0&\ldots&-n&2n&-n&0\\
0&0&0&0&\ldots&0&-n&2n&n\\
0&0&0&0&\ldots&0&0&n&n-1\esm
\eeq
Thus we obtain that $\Ap_n^\perp$ is even if $n$ is odd, and, for $n$ even
\beq\label{lat2}
(\calAp(n)^\perp)^\ev \cong \la 2n\ra\oplus \sfU(n)\oplus (A_{n-1}^\vee(n))^\ev,
\eeq
where $(A_{n-1}^\vee(n))^\ev$ is defined by the matrix \eqref{mat1}  in which the submatrix $\bsm 2n&n\\
n&n-1\esm$ is replaced with $\bsm 2n&2n\\
2n&4n-4\esm$.

Note that the Apollonian packing in dimension $n = 1$ exists but it is not a Boyd-Maxwell sphere packing. We have
$$\calAp_1 \cong \la 1\ra \oplus U(2), \quad \calAp_1^{\ev} \cong (\sfA_1\oplus U)(2).$$
The lattice $\calAp_1^{\ev}(1/2) = \sfA_1\sfU$ is from Nikulin's list of integral even hyperbolic quadratic lattices of rank 3 and signature $(2,1)$ such that the subgroup of the orthogonal group generated by reflections $s_v, (v,v) = 2$ is of finite index (a 2-reflexive lattice) \cite{Nikulin2}. Each such lattice is a sublattice of $\sfA_1\oplus \sfU$.  
According to Nikulin's classification of even 2-reflexive hyperbolic lattices of rank $\ge 5$, none of the lattices 
$\calAp(n)^{\ev}$ or $\calAp(n)^{\ev}$ is 2-reflexive \cite{Nikulin1}. 

If $n= 2$, the lattices $\calAp_2$ and $\calAp_2^\perp$ are isomorphic. According to Vinberg's list of 2-reflexive even hyperbolic lattice of rank 4 \cite{Vinberg2}, the lattice $\calAp_2^{\ev}$ is  not 2-reflexive.

\end{example}

Note that all examples from  \cite{Boyd2} are examples  of integral Boyd-Maxwell sphere packings.

\section{Orbital counting}
Let $A$ be a subset of the Euclidean space $\bbR^n$. Recall that the \emph{Hausdorff dimension} $\delta(A)$ of $A$ is defined to be the infimum for all $s\ge0$ for which
\beq\label{hd1}
\mu_s(A) = \inf_{A\subset \cup_jB_j}(\sum_jr(B_j)^s) = 0.
\eeq
Here $(B_j)$ is a countable set of open balls of radii $r(B_j)$ which cover $A$.

For example, if the Lebesque measure of $A$ is equal to $0$, then \eqref{hd1} holds for $s =n$, hence $\delta(A) \le n$. The Hausdorff dimension coincides with the Lebesque measure if the latter is positive and finite. A countable set has the Hausdorff measure equal to zero. Also it is known that the topological dimension of $A$ is less than or equal to its Hausdorff dimension.

The Hausdorff dimension is closely related to the \emph{fractal dimension} of a \emph{fractal set} $A$, i.e. a set  that can be subdivided in some 
finite number $N(\lambda)$ of subsets, all congruent (by translation or rotation) to one another and each equal to a scaled copy of 
$A$ by a linear factor $\lambda$. It is defined to be equal to $\frac{\log N(s)}{\log(1/\lambda)}$. For example,   
the Cantor set consists of two parts $A_1$ and $A_2$ (contained in the interval $[0,1/3]$ and $[1/3,1]$), each rescaled version of the set 
with the scaling factor $1/3$. Thus its fractal dimension is equal to $\log 2/\log 3 < 1$. It is clear that the Hausdorff 
dimension of a bounded fractal set of diameter $D=2R$ is less than or equal than the fractal dimension. In fact, such a set can be covered by $N(\lambda)$ 
balls of radius $\lambda R$, or by $N(\lambda)^2$ balls of radius $\lambda^2R$. Since $\lambda < 1$, we get 
$\mu_s(A) = \lim_{k\to \infty}\frac{N(\lambda)^k}{(\lambda^kR)^s}$ which is zero if $N(\lambda)\lambda^s <1$ or $s > \log N/\log(1/\lambda)$.
In fact, the Hausdorff dimension of the Cantor set coincides with its fractal dimension $\log 2/\log 3$.

By a theorem of D. Sullivan \cite{Sullivan}, for a geometrically finite non-elementary discrete group $\Gamma$, the Hausdorff dimension $\delta_\Gamma$ 
of $\Lambda(\Gamma)$ is positive and coincides with the \emph{critical exponent} of $\Gamma$ equal to 
\beq\label{critical}
 \inf\{s>0:\sum_{g\in \Gamma}e^{-sd(x_0,g(x_0))}<\infty\},
\eeq
where $x_0$ is any point on $\bbH^n$. Using this equality Sullivan shows that 
\beq\label{sul1}
\delta_\Gamma = \overline{\underset{T\to \infty}{\lim}}\frac{\log N_T}{R},
\eeq
where $N_T$ is the number of orbit points  $y$ with hyperbolic distance from $x_0$ less than or equal than $R$. He 
further shows in \cite{Sullivan2}, Corollary 10, under the additional assumption that $\Gamma$ has no parabolic fixed points, 
 that there exists constants $c,C$ such that, as $T\to \infty$, 
\beq\label{sul2}
Ce^{T\delta_\Gamma}\le N_T \le Ce^{T\delta_\Gamma}.
\eeq
In particular, asymptotically, as $T\to \infty$,
$$N_T \sim c(T)e^{T\delta_\Gamma},$$
where $\underset{T\to\infty}{\lim}\frac{c(T)}{T} = 0$.

 If $\delta_\Gamma > \frac{1}{2}(n-1)$,  then P. Lax and R. Phillips show that, for any geometrically finite 
 non-elementary discrete group $\Gamma$, the function $c(T)$ is a constant depending only on $\Gamma$. When $\Gamma$ is of finite covolume, then $\delta_\Gamma$ is known to be equal to $n-1$, and the result goes back to A. Selberg. 
 
 The assumption on $\delta_\Gamma$ has been lifted by T. Roblin \cite{Roblin}. Applying \eqref{dist1}, we will be able to  obtain the assertion of Theorem \ref{intro:thm2} from the introduction in the case when 
 $C^2< 0$. To do the remaining cases where $(e,e)$ is non-negative   we need to replace the family of hyperbolic balls with another  family of sets of growing volume. 
 
Let $e$ be a nonzero vector in $\bbR^{n,1}$. We assume that $(e,e) = \pm 1$ if $(e,e)\ne 0$ and we fix an isotropic $e'$ with $(e,e') = -1$ if $(e,e) = 0$. The group $G = \SO(n,1)_0$ acts in the projective space $\bbP(\bbR^{n,1})$ with three orbits $G\cdot [e]$. Let $G_{[e]}$ be the stabilizer subgroup of $[e]$ in $G$. If $(e,e) < 0$, then $G_{[e]}$ is a maximal compact subgroup in $K\subset G$ isomorphic to $\SO(n)$ and the orbit $G\cdot [e]$ is the homogeneous space of left cosets $G/K$ isomorphic to $\bbH^n$. 

If $(e,e) = 0$, then $G_{[e]}$ is a parabolic subgroup of $G$. If we identify $\sfQ\setminus \{[e]\}$ with the Euclidean space $E^{n-1}$ (via the projecting from $[e]$), we obtain a surjective homomorphism from $G_{[e]}$ to the group $P$ of 
affine orthogonal transformations (motions) of the Euclidean space $E^{n-1}$. This homomorphism splits by a subgroup 
$G_{ce}$ of $G_{[e]}$ that stabilizes any nonzero vector $ce$ on the line $[e]$. The kernel of the homomorphism is the group $A_l$ of hyperbolic translations along the hyperbolic line $l$ corresponding to the subspace spanned by $e$ and $e'$. This shows that 
$G_{[e]} = O(n-1)\cdot A_l\cdot N_{[e]},$
where $O(n-1)$ is the group of rotations around the line $l$ and $N_{[e]} \cong \bbR^{n-1}$ is the normal subgroup of translations of $P$, called a \emph{horospherical subgroup} of $G$. The decomposition of $G_{[e]}$ from above 
is induced by the  \emph{Iwasawa decomposition} $G = G_x\cdot A_l\cdot N_{[e]}$ of $G$, where $x$ is any point on $l$. 
The group $G_{[e]}$ acts transitively on the quadric $\sfQ$ and on $\bbH^p$. 

If $(e,e) > 0$, then $G_{[e]}\cong \SO(n-1,1)_0$, the homogeneous space $G_{[e]}\backslash G$ is known as a \emph{de Sitter space}. By taking the orthogonal complement of the line spanned by $e$, we identify the points of the de Sitter space with oriented hyperplanes $H_\frake$ in $\bbH^n$. For example, if $n = 2$, the de Sitter space is the set of oriented geodesic lines in the hyperbolic plane.

 Let us consider the \emph{pencil of geodesic lines} $\calP(e)$ defined by  indefinite planes $U\subset \bbR^{n,1}$ containing $e$. If $(e,e) < 0$, then $\calP(e)$ consists of geodesic lines containing the point $x_0 = [e]\in \bbH^n$. It is called an \emph{elliptic pencil}. If $(e,e) > 0$, it consists of parallel geodesic lines perpendicular to the hyperplane $H_e$. It is called a \emph{hyperbolic pencil}. Finally, if $(e,e) = 0$, it consists of geodesic lines whose closure in $\overline{\bbH^n}$ contains the point $[e]$. The pencil is called \emph{parabolic pencil} with center at $[e]$.

The parametric equation of the geodesic line from the pencil $\calP(e)$ is equal to 
\[
\gamma(t) = \begin{cases}
      v\sinh t+e\cosh t, (v,e) = 0, (v,v) = 1& \text{if } (e,e) = -1, \\
      v\exp(t)+e\sinh t, (v,e) = 1, (v,v) = -1& \text{if } (e,e) = 0,\\
      v\cosh t-e\sinh t, (v,e) = 0, (v,v) = -1&\text{otherwise}.
\end{cases}
\]
(see \cite{Vinberg}, Chapter 4, 2.3). In the elliptic  case, we have
$(\gamma(t), e) = -\cosh t$, so that $d(\gamma(t),[e]) = t$. So, moving along the geodesic line for the distance $t \ge 0$, we get the set of points in $\bbH^n$ equidistant from $[e]$. 
This is a \emph{geodesic sphere} 
$$H_e^{t}: = \{[v]\in \bbH^n:-(v,e) = \cosh t\} = \{[v]\in \bbH^n:d([v],[e]) = t\}$$
 with the center at $x_0 = [e]$. The group $G_{[e]}\cong \SO(n)$ acts transitively on each $H_e^t$ and, projecting from $[e]$ along the lines from the pencil $\calP(e)$ identifies each geodesic sphere with the absolute as homogenous spaces with respect to $G_{[e]}$.  So, the orbits of $G_{[e]}$ in $\bbH^n$ are Riemannian homogeneous spaces of constant positive curvature.
 
If $(e,e) = 0$, we have $(\gamma(t),e) = \exp(t)$. The orbits of $N_{[e]}$ are \emph{horospheres}
$$H_e^t:= \{[v]\in \bbH^n:-(v,e) = \exp(t)\}.$$ 
In the vector model of $\bbH^n$, a horosphere  is equal to the intersection of $\bbH^n$ with a sphere in the Euclidean space tangent to $\sfQ$ at the point $[e]$. In the Klein model, they are ellipsoids tangent to the boundary of the ball. Any horosphere inherits a Riemannian metric of zero curvarure and hence isomorphic to an Euclidean space.

If $(e,e) > 0$, we have $(v,e) = (\gamma(0),e) = 0$, i.e. $[\gamma(0)]\in H_e$. Since $(\gamma(t),e) = -\sinh t$, after 
time $t$ we move from a point on the hyperplane $H_e$ to a point $[\gamma(t)]$ on the hypersurface 
$$H_e^{t} = \{[w]\in \bbH^n: -(w,e) = \sinh t\}.$$
The hypersurfaces of this form are called \emph{equidistant hypersurfaces}. They are orbits of the group $G_{[e]}$ and 
inherit a Riemannian metric of constant negative curvature. In the Klein model, they are ellipsoids whose closures in 
$\overline{\bbH^n}$ intersect the closure $\bar{H}_e$ of the hyperplane $H_e$ along its boundary. For example, if $n = 2$, they are ellipses tangent to the absolute at the two points in which the closure of the geodesic line $H_e$ intersects the absolute.

In each of the three cases, let us consider the following sets $B_T(e)$.

If $(e,e) = -1$, 
$$B_T(e) = \bigcup_{t=0}^TH_e^t = \{[v]\in \bbH^n:1\le -(v,e) \le \cosh T\}.$$
This is just a hyperbolic closed ball with center at $[e]$ and radius $T$.

If $(e,e) = 0$,
$$B_T(e) = \bigcup_{t=0}^TH_e^t = \{[v]\in \bbH^n:1\le |(v,e)| \le \exp(T)\}.$$

If $(e,e) = 1$,
$$B_T(e) = \bigcup_{t=0}^TH_e^t = \{[v]\in \bbH^n: |(v,e)| \le \sinh(T)\}.$$

Fix a point $x_0 = [v_0]\in \bbH^n$ and let $K = G_{x_0}$ so that $G/K = \bbH^n$. Let $H = \{1\}$ (resp. $N_{[e]}$, resp. $G_{[e]}$ if $(e,e) = -1$ (resp. $(e,e) = 0$, resp. $(e,e) = 1$). 

Consider the following subsets of $G/H$ 
$$\calS_T(e,x_0):= G_{x_0}\cdot A_T\cdot H/H,$$
where $A_T = \{a_t\in A:0\le t\le \log T\}.$

The following result from \cite{OS}, Theorem 1.2 and \cite{OM}, Corollary 7.14 (see a nice survey of some of these results in \cite{Oh}) is crucial for our applications.

\begin{theorem}\label{om} Let $\Gamma$ be a torsion free geometrically finite non-elementary discrete  subgroup of $\SO(n,1)_0$. Let $[e]\in \bbP(\bbR^{n,1})$. If $(e,e) \ge 0$, assume that the orbit
$\rmO_\Gamma([e])$ is a discrete set and also that $\delta_\Gamma > 1$ if $(e,e) > 0$. Then
$$\lim_{T\to \infty}\frac{\Gamma\cdot H/H\cap \calS_T(e,x_0)}{T^{\delta_\Gamma}} = c_{\Gamma,x_0,[e]},$$
where $c_{\Gamma,x_0,[e]}$ is a positive constant that  depends only on $\Gamma,x_0,[e]$.  
\end{theorem}

Note that the results from \cite{OM} also give error terms.

\begin{corollary} Keep the same assumptions. Fix a point $x_0\in \bbH^n$. Then
$$\lim_{t\to \infty}\frac{\#\{[v]\in O_{\Gamma}(x_0):|(e,v)| \le T\}}{T^{\delta_\Gamma}} = c_{\Gamma,x_0,[e]},$$
where $c_{\Gamma,x_0,[e]}$ is a positive constant that  depends only on $\Gamma,x_0,[e]$.  
\end{corollary}

\begin{remark} The assumption $\delta_\Gamma > 1$ is not needed if the group $\Gamma\cap N_{[e]}$ acts on $\partial H_e$ without parabolic fixed points. Also, the assumption that $\Gamma$ is torsion free made in loc.cit. is needed only for an explicit formula for the constant. 
\end{remark}

\begin{remark} In the case when $(e,e)< 0$ (resp. $(e,e) = 0$ and $n = 3$), Theorem \ref{om} follows from \cite{OS}, Theorem 1.2 (resp. \cite{KO}, Theorem 2.10) that gives an asymptotic of the number of orbit points in a ball $\{[v]\in \bbH^n:||v|| \le T\}$, where $||x||$ is the Euclidean norm. 
One has only use that if $v = (x_0,\ldots,x_n)$ with $\sum_{i=0}^nx_i^2 \le T^2$ and $-x_0^2+\sum_{i=1}^nx_i^2 = -1$ (resp. $=0$), then 
$(v,(1,\ldots,0))^2 = x_0^2 \le (T^2+1)/2$ (resp. $\le T^2$). 

In the case  $(e,e) > 0$, and $\Gamma$  and  $\Gamma\cap G_{[e]}$ are of finite covolume (the second condition is 
always satisfied if $n\ge 3$ \cite{DM}),  Theorem \eqref{om} follows from 
\cite{DRS}, \cite{EM}. In this case $\delta_\Gamma = n-1$.
\end{remark}

\section{Algebraic geometrical realization}
Let $X$ be a smooth projective algebraic surface over an algebraically closed field $\Bbbk$. 
The group of its automorphisms defines a natural left action $\gamma\mapsto (g^{-1})^*(\gamma)$ on its lattice $\Num(X)$ of algebraic cycles modulo numerical equivalence. As usual, we denote 
its rank by $\rho(X)$, this is the \emph{Picard number} of $X$. By the Hodge  Index Theorem, the intersection form on algebraic  cycles defines a symmetric bilinear form on $\Num(X)$ of signature $(1,n)$. We denote by $V_X$ the corresponding real vector space $\Num(X)_\bbR$ equipped with the induced symmetric bilinear form of signature $(1,n)$.  We assume that the image $\Aut(X)^*$ of the homomorphism $\Aut(X)\to \Or(\Num(X))$ is infinite. It is known that this could happen only if $X$ is an abelian surface, a K3 surface, an Enriques surface, or a rational surface.  Moreover, $n$ must be larger than or equal to $2$, and in the last case, $n\ge 9)$. All of this is rather well-known, see, for example, \cite{DolgachevR}. Let $\Gamma$ be an infinite subgroup of $\Aut(X)^*$. Since the automorphism group preserves the connected component of $\{v\in \Num(X)_\bbR:(v,v) > 0\}$ containing an ample class on $X$, we obtain that $\Gamma$ is  a discrete  subgroup in $\Or(V_X) \cong \rmO(n,1)$. Also, since $\Gamma$ preserves an integral lattice, its orbits are discrete sets. Thus, we can apply Theorem \ref{om} to obtain the following.

\begin{theorem}\label{AA} Assume that $\Gamma$ is a geometrically finite non-elementary subgroup of $\Aut(X)^*$.  Let
$H,C$ be effective numerical divisor classes in $\Num(X)'$ such that $H$ is ample.  Then
\beq
\lim_{t\to \infty}\frac{\#\{C'\in \rmO_\Gamma(C):(H,C')/(H,C) \le T\}}{T^{\delta_\Gamma}} = c_{\Gamma,C,H},
\eeq
where $c_{\Gamma,H,C}$ is a positive constant depending only on $\Gamma,[H],[C]$  but does not depend on $C^2$ and $H^2$.
\end{theorem}

Now let us give some examples where we can apply the previous result. First, we note that we may always consider $\Gamma$ up to a finite group, i.e. we either consider its subgroup of finite index, or extend it by a finite group, or do both. It affects only the constant in the asymptotic formula.

Let $X$ be a complex algebraic K3 surface. It is known that the subgroup of $\Or(\Num(X))$ generated by transformations $g^*, g\in \Aut(X)$ and reflections $s_r, (r,r) = -2$ is a subgroup of finite index  (this is also true if the characteristic $p \ne 2$ \cite{LM}).  Thus, if $\Num(X)$ does not contain classes $r$ with $(r,r) = -2$, then $\Aut(X)$ is a Kleinian group of finite covolume.\footnote{By Riemann-Roch theorem, such a divisor class is either effective or anti-effective, and one of its irreducible component is a smooth rational curve.} The Hausdorff dimension $\delta_\Gamma$ in this case is equal to $n-1$. So, it is more interesting to realize proper thin Kleinian groups of automorphisms.

Assume that we find a surface $X$ and vectors $\alpha_1,\ldots,\alpha_r\in \Num(X)$ such that 
$\frac{2}{(\alpha_i,\alpha_i)}\alpha_i\in \Num(X)^\vee$, so that the reflections 
$s_{\alpha_i}$ 
define orthogonal transformations of $\Num(X)$. Moreover, assume that there exist automorphisms $g_i\in \Aut(X)$ 
such that 
$g_i^* = s_{\alpha_i}$. Let $P$ be the convex polyhedron with normal vectors 
$\frake_i = \frac{1}{(\alpha_i,\alpha_i)^{1/2}}\alpha_i$. Assume that it is a Coxeter polyhedron of level $\le 2$ with the 
reflection group $\Gamma_P$ generated by $s_{\alpha_1},\ldots, s_{\alpha_r}$. Then the limit set of $\Gamma_P$ is a Boyd-Maxwell sphere packing and its Hausdorff dimension $\delta_{\Gamma_P}$ is equal to the critical exponent \eqref{crit}. Now we can apply Theorem \ref{AA}.  

\begin{example}\label{enr} Let $M = \calAp(2)^{\ev} \cong (\sfA_1\oplus \sfU)(2)$ be the even sublattice of the Apollonian lattice and 
$\Ap_2$ be the Apollonian group. Let $\omega_1=\frake_1,\ldots,\omega_4 = \frake_4$ be a basis in $\calAp(2)$ with Gram matrix 
equal to $G = \circulant(1,-1,-1,-1) = (g_{ij})$. The lattice $M(-1/2) \cong \sfA_1(-1)\oplus \sfU$ can be primitively embedded in the lattice 
$\Num(X) \cong \sfU\oplus E_8(-1)$, where $X$ is an Enriques surface. We identify the image with $M(-1/2)$. 
 The group $\Ap_2$ contains a subgroup of index 2 that consists of isometries of $M(-1/2)$ that extend to the whole lattice acting identically on the orthogonal complement. Thus
the subgroup $\Gamma = \Ap_2'\cap \Or(\Num(X)'$ is realized by automorphisms of $X$. One can show that the index of 
$\Gamma$ in $\Ap_2$ is equal to  4. Now we can apply Theorem \ref{AA} with 
$\delta_\Gamma = \delta_{\Ap_2}$.

Of course, one can also realize $\calAp(2)(-1)$  as the Picard lattice  of a K3 surface $X$. Since the lattice does not represent $-2$,  a subgroup of finite  index of $\Ap_2$ is realized as a group of automorphisms of a K3 surface.
according to Vinberg's classification of 2-reflexive lattices of rank 4 \cite{Vinberg2}, the 2-reflection group of the lattice 
$\calAp(2)$ is of infinite index in the orthogonal group of the lattice. Thus $\Ap_2$ is a thin group geometrically finite group of automorphisms of $X$.

The same is true for the group $\Ap_3$; it contains a subgroup of finite index that can be realized as a 
thin geometrically finite  group of automorphisms of either an Enriques or a K3 surface.

\end{example}

\begin{example} Consider the Coxeter group $\Gamma(a,b,c)$ of a Coxeter polyhedron $P$ in $\bbH^2$ with Gram matrix 
$\bsm 1&-a&-b\\
-a&1&-c\\
-b&-c&1\esm$, where $a,b,c$ are rational numbers $ \ge 1$. If $a = b = c = 1$, the fundamental domain is an ideal 
triangle. If $a,b,c > 2$,  the fundamental triangle is a pair of pants as on the following picture.

\xy (-10,-35)*{};
(50,0)*{\color{blue}\cir<50pt>{}};
(57,17)*{\color{red}\cir<30pt>{d^ur}};
(25,0)*{\color{red}\cir<50pt>{dr_dl}};
(66.5,-7)*{\color{red}\cir<35pt>{l^dr}};
(50,-25)*{\textrm{Figure 7: Hyperbolic triangle with no vertices}}
\endxy

Let $L$ be the quadratic form defined by the integral matrix $A = -eG(P)$, where $e$ is the exponent of the 
Coxeter matrix. Let $L^\vee$ be the dual lattice. Suppose we can embed $M = \integ(L^\vee)^{\ev}$ in $\Num(X)$ of some K3 surface and realize the reflection group $\Gamma_P$ as a group of automorphisms of $X$. Then we can apply Theorem \ref{AA} to any nef divisor class contained in the sublattice $M$. 

For example, let us consider the case $(a,b,c) = (a,a,1)$. The fundamental triangle has one ideal point and looks like in the following figure \ref{fig7}.

\begin{figure}
\begin{center}
\xy (-60,-35)*{};
(0,0)*{\color{blue}\cir<49pt>{}};
(17,17)*{\color{red}\cir<48pt>{d^r}};
(17,-17)*{\color{red}\cir<48pt>{l^d}};
(-25,0)*{\color{red}\cir<50pt>{dr_dl}};
(10,-25)*{\textrm{Figure 8: Hyperbolic triangle with one ideal point}}
\endxy
\label{fig7}
\end{center}
\end{figure}

Let $\calH^2  \to \bbH^2$ be the map from the upper-half plane to the unit disk given by the map $z\mapsto \frac{z-i}{z+i}$. One can show that the pre-image of the sides of our triangle are the lines $x = 1, x = -1$ and the upper half-circle of radius $r = 1/a$ with center at the origin. Recall that $a\le -1$ so that the half-circle is between the vertical lines.  Let us re-denote our group by $\Gamma_r$. It was shown by C. McMullen in \cite{McMullen} that as $r\to 0$, we have
$$
\delta_{\Gamma_r} = \frac{r+1}{2}+O(r^2),$$
while for $r\to 1$, we have
$$\delta_{\Gamma_r} \sim 1-\sqrt{1-r}.$$
 We have 
$$G(P)^{-1} = \frac{1}{4a^2}\begin{pmatrix}0&-2a&-2a\\
-2a&a-1&-a-1\\
-2a&-a-1&a-1\end{pmatrix}$$
If $a > 1$ is odd, the matrix $2a^2G(P)^{-1}$ defines an even integral lattice $L_r$. The  group $\Gamma_r$ acts on this lattice as a reflection group in the sides of the triangle. The lattice $L_r(-2)$ is realized as the Picard lattice of a K3 surface $X$. Since $L_r(-2)$ does represent $-2$, the surface $X$ does not contain smooth rational curves. This implies that a subgroup of finite index of $\Gamma_r$ acts on $X$ by automorphisms. So, we may apply Theorems \ref{AA} with the Hausdorff dimension computed by McMullen. Note that when $r = a = 1$, the lattice $L_r$ is isomorphic to the Apollonian lattice $\calAp_1$ and the group $\Gamma_1$ is isomorphic to the Apollonian group $\Ap_1$. 

Let us give another example of a realizable group $\Gamma(a,b,c)$. It is taken  from \cite{Baragar1}. Let $X$ be a K3 surface defined over an algebraically closed field of characteristic $\ne 2$ embedded  in $\bbP^2\times \bbP^2$ as a complete intersection of hypersurfaces  of multi-degree $(1,1)$ and $(2,2)$. Let $p_1,p_2:X\to \bbP^2$ be the two projections. They are morphisms of degree 2 branched along a plane curve $B_i$ of degree 6. We assume that $B_1$ is nonsingular and $B_2$ has a unique double point $q_0$ so that the fiber $p_2^{-1}(q_0)$ is a smooth rational curve $R$ that is mapped isomorphically under $p_1$ to a line. We assume that $X$ is general with these properties. More precisely, we assume that $\Pic(X)$ has a basis $(h_1,h_2,r)$, where $h_i = p_i^*(\textrm{line})$ and $r$ is the class of $R$. The intersection matrix of this  basis is equal to 
$$\begin{pmatrix} 2&4&1\\
4&2&0\\
1&0&-2\end{pmatrix}.$$
 Let $s = p_1^{*}((p_1)_*(r))-r = h_1-r$. It is a class of smooth rational curve $S$ on $X$. The pre-image of the pencil of lines through $q_0$ is an elliptic pencil $|F|$ on $X$ with $[F] = h_2-r$. The curve $S$ is a section 
of the elliptic fibration defined by the linear system $|F|$ and the curve $R$ is its 2-section that intersects $S$ with multiplicity 3.  Consider the following three automorphisms of $X$. The first two 
$\Phi_1$ and $\Phi_2$ are defined by the birational deck transformations of the covers $p_1$ and $p_2$. The third 
one $\Phi_3$ is defined by the negation automorphism of the elliptic pencil with the group law defined by the choice of $S$ as the zero section. 

It is easy to compute the matrix of each $\Phi_i$ in the basis $(f,s,r) = ([F],[S],[R])$ with the Gram matrix
$$\begin{pmatrix}0&1&2\\
1&-2&3\\
2&3&-2\end{pmatrix}.$$
 We have 
$\Phi_1^*(s) = r, \Phi_1^*(r) = s$ and $f' = \Phi_1^*(f) = af+bs+cr$. Since $\Phi_i^2$ is the identity and $(f,f) = 0$, we get 
$a = -1$ and $b = c$. Since  $(f',s) = (f,r) = 2$, we easily get $b = c = 3$. The matrix of $\Phi_1$, and similarly obtained matrices of $\Phi_2$ and $\Phi_3$ are as follows. 
$$A_1 = \begin{pmatrix} -1&0&0\\
3&0&1\\
3&1&0\end{pmatrix},\quad A_2 = \begin{pmatrix}1&4&0\\
0&-1&0\\
0&1&1\end{pmatrix}, \quad A_3 = \begin{pmatrix}1&0&14\\
0&1&4\\
0&0&-1\end{pmatrix}.$$
The transformations $\Phi_1' = \Phi_1\circ\Phi_2\circ \Phi_1, \Phi_2,\Phi_3$ are the reflections with respect to the vector $\alpha_i$, where
$$\alpha_1 = -4f+13s+10r,\ \alpha_2 = 4f-2s+r, \ \alpha_3 = 7f+2s-r.$$
The Gram matrix of the vectors $\alpha_1,\alpha_2,\alpha_3$ is equal to 
$$G = \begin{pmatrix}-22&143&220\\
143&-22&22\\
220&22&-22\end{pmatrix} = -22\begin{pmatrix}1&-\frac{13}{2}&-10\\
-\frac{13}{2}&1&-1\\
-10&-1&1\end{pmatrix}.$$
So, the group generated by $\Phi_1',\Phi_2,\Phi_3$ coincides with the triangle group $\Gamma(\frac{13}{2},10,1)$. The fundamental triangle  $P$ has one ideal vertex. The reflection group $\Gamma_P$ is a subgroup of infinite index of the 
group $\Gamma$
of automorphisms of $X$ generated by $\Phi_1,\Phi_2,\Phi_3$. 
Baragar proves that $\Gamma$ is isomorphic to $\Aut(X)$ (for sufficiently general $X$). He  finds  the following bounds for $\delta_{\Gamma}$
$$.6515< \delta_{\Gamma} < .6538.$$
This implies that
$$\delta_{\Gamma_P} < .6538.$$
\end{example}

\begin{example} This is again due to Baragar \cite{Baragar2}. We consider a nonsingular hypersurface $X$ in $\bbP^1\times \bbP^1\times \bbP^1$ of type $(2,2,2)$. It is a K3 surface whose Picard lattice contains the Apollonian lattice $\calAp(1)$. If $X$ is general, then the Picard lattice coincides with this lattice. We assume that one of the projections 
$p_{ij}:X\to \bbP^1\times \bbP^1$, say $p_{12}$, contains the whole $\bbP^1$ as  its fiber over some point $q_0\in \bbP^1\times \bbP^1$ . All the projections are degree 2 maps. Let $F_i, i = 1,2,3,$  be the general fibers of the projections $p_i:X\to \bbP^1$. Each $F_i$ is an elliptic curve whose image  under the map $p_{jk}$ is a divisor of type $(2,2)$. Let $f_1,f_2,f_3,r$ be the classes of the curves $F_1,F_2,F_3,R$. We assume that $X$ is general with these properties so that $\Pic(X)$ is generated by these classes. The Gram matrix of this basis is equal to 
$$\begin{pmatrix}0&2&2&0\\
2&0&2&0\\
2&2&0&1\\
0&0&1&-2\end{pmatrix}.$$
It is easy to see that 
$$\Pic(X) \cong \sfU\oplus \bsm -4&2\\
-2&-8\esm.$$
According to Vinberg's classification of 2-reflective hyperbolic lattices of rank 4 \cite{Vinberg}, the Picard lattice is not 2-reflective. Hence the image of the group $\Aut(X)$ in $\Or(\Pic(X))$ is of infinite index. 

Let $\Phi_{ij}$ be the automorphisms of $X$ defined by the deck transformations of the projections $p_{ij}$. Let $\Phi_4'$ be defined as the transformation $\Phi_3$ in the previous example with respect to the elliptic pencil $|F_3|$ with section $R$. The transformation $\Phi_{12}^*$ leaves the vectors $f_1,f_2,r$ invariant, and transforms $f_3$ to $2f_1+2f_2-f_3-r$. Thus $\Phi_{12}^*$ is the reflection with respect to the vector $\alpha_1 = -2f_1-2f_2+2f_3+r$. 

The transformation $\Phi_{13}^*$ leaves $f_1,f_3$ invariant and transforms $r$ in $r' = f_1-r$. It also transforms 
$f_2$ to some vector $f_2' = af_1+bf_2+cf_3+dr$. Computing $(f_2',f_1) = (f_2,f_1), (f_2',f_3) = (f_2,f_3), (f_2',r) =
(f_2,f_1-r)$, we find that $f_2' = 2f_1-f_2+2f_3$. Similarly, we find that 
$\Phi_{23}^*(f_2) = f_2, \Phi_{23}^*(f_3) = f_3, \Phi_{23}^*(r) = f_2-r$ and $\Phi_{23}^*(f_2) = -f_1+2f_2+2f_3$. 

It follows from the definition of a group law on an elliptic curve that 
$$\Phi_4'{}^*(f_3) = f_3,\  \Phi_4'{}^*(r) = r, \ \Phi_4'{}^*(f_i) = -f_i+8f_3+4r,\  i = 1,2$$
Consider the transformations
$$\Phi_1 = \Phi_{12},\  \Phi_2 = \Phi_{13}\circ \Phi_{12}\circ \Phi_{13},\  
\Phi_3 = \Phi_{23}\circ \Phi_{12}\circ \Phi_{23}, \ \Phi_4 = \Phi_4'\circ \Phi_{12}\circ \Phi_4'.$$
These transformations act on $\Pic(X)$ as the reflections with respect to the vectors 
\begin{eqnarray*}
\alpha_1 &=& -2f_1-2f_2+2f_3+r,\\
 \alpha_2 &=& \Phi_{13}^*(\alpha_1) = -5f_1+2f_2-2f_3-r,\\
 \alpha_3 &=& \Phi_{23}^*(\alpha_1) = 
2f_1-5f_2-2f_3-r,\\
\alpha_4 &=& \Phi_4(\alpha_1) = 2f_1+2f_2-30f_3-15r.
\end{eqnarray*}
The Gram matrix of these four vectors is equal to 
$$\begin{pmatrix}-14&14&14&210\\
14&-14&84&182\\
14&84&-14&182\\
210&182&182&-14\end{pmatrix} = -14\begin{pmatrix}1&-1&-1&-15\\
-1&1&-6&-13\\
-1&-6&1&-13\\
-15&-13&-13&1\end{pmatrix}
$$
 Let $P$ be the Coxeter polytope defined by this matrix. The  Coxeter group $\Gamma_P$ is generated by the reflections $\Phi_i^*, i = 1,2,3,4$.
 
 Baragar proves that the automorphisms  $\Phi_{ij}$ and $\Phi_4'$ generate a subgroup $\Gamma$ of $\Aut(X)$ of 
 finite index. His computer experiments suggest that 
 $$1.286< \delta_{\Gamma} < 1.306.$$
 Our reflection group $\Gamma_P$ generated by $\Phi_1,\ldots,\Phi_4$ is of infinite index in $\Gamma$. So, we obtain
 $$\delta_{\Gamma_P} < 1.306.$$
\end{example}

\begin{example} We consider a general Coble rational surface \cite{Cantat}. The orthogonal complement of the canonical class in $\Num(X)$ is isomorphic to $\sfU\oplus \sfE_8(-1)$. One can prove that the automorphism group of $X$ 
is isomorphic to the automorphism  group of a general Enriques surface \cite{Cantat} (true in any characteristic). We do the same, as in Example \ref{enr} to realize $\Gamma$ as an automorphism group of $X$. Since $X$ is rational, it gives a realization of $\Gamma$ as a 
group of Cremona transformations of $\bbP^2$. Taking the class of a line, we obtain the asymptotic of the growth of the function $\deg \Phi$, where $\deg \Phi$ is the algebraic degree of a Cremona transformation $\Phi$ from $\Gamma$.
\end{example}

\begin{remark}\label{new2} Taking into account Remark \ref{new}, we see that, for any even hyperbolic lattice of rank $> 4$ that contains primitively the lattice $M = \sfU\oplus \sfA_1(-1)\oplus \sfA_1(-1)$, the orthogonal group $\Or(L)$ contains finitely generated subgroups which are not geometrically finite. Embedding $M$ primitively in the lattice $\sfU\oplus \sfE_8(-1)$, we obtain that a general Enriques surface contains finitely generated groups of automorphisms which are not  geometrically finite.
\end{remark}

 \end{document}